\documentclass[a4paper,10pt,psamsfonts]{amsart}

\usepackage[utf8]{inputenc} 

\usepackage{amsmath}
\usepackage{amstext}
\usepackage{amsfonts}
\usepackage{amsthm}
\usepackage{amssymb}

\usepackage[bf,SL,BF]{subfigure}
\usepackage{graphicx,epstopdf}

\usepackage{caption}
\usepackage{listings}
\usepackage{hyperref}
\hypersetup{colorlinks}

\usepackage{color}

\usepackage{changes}

\usepackage{lineno}
\usepackage{algorithm2e}

\usepackage{epsfig}
\usepackage{pst-all} 
\usepackage{pst-plot} 
\usepackage[space]{grffile}

\newcommand{\triple}{{\vert\kern-0.25ex\vert\kern-0.25ex\vert}}

\theoremstyle{plain}

\newcommand{\be}{\begin{equation}}
\newcommand{\ee}{\end{equation}}

\newtheorem{theorem}{Theorem}[section]
\newtheorem{lemma}[theorem]{Lemma}
\newtheorem{proposition}[theorem]{Proposition}
\newtheorem{corollary}[theorem]{Corollary}

\theoremstyle{definition}

\begin{document}

\title[Fast, sketched FEM]{A sketched finite element method for elliptic models}

\author[R.~Lung]{Robert Lung}
\address{Robert Lung\\
School of Engineering\\
University of Edinburgh\\
UK}
\email{robert.lung@ed.ac.uk}

\author[Y.~Wu]{Yue Wu}
\address{Yue Wu\\
Mathematical Institute\\
University of Oxford\\
Oxford\\
UK}
\email{yue.wu@maths.ox.ac.uk}

\author[D.~Kamilis]{Dimitris Kamilis}
\address{Dimitris Kamilis\\
School of Engineering\\
University of Edinburgh\\
EH9 3JL Edinburgh\\
UK}
\email{d.kamilis@ed.ac.uk}

\author[N.~Polydorides]{Nick Polydorides}
\address{Nick Polydorides\\
School of Engineering\\
University of Edinburgh\\
EH9 3JL Edinburgh\\
UK \& The Alan Turing Institute\\
London, UK}
\email{n.polydorides@ed.ac.uk}


\begin{abstract}
We consider a sketched implementation of the finite element method for elliptic partial differential equations on high-dimensional models. Motivated by applications in real-time simulation and prediction we propose an algorithm that involves projecting the finite element solution onto a low-dimensional subspace and sketching the reduced equations using randomised sampling. We show that a sampling distribution based on the leverage scores of a tall matrix associated with the discrete Laplacian operator, can achieve nearly optimal performance and a significant speedup. We derive an expression of the complexity of the algorithm in terms of the number of samples that are necessary to meet an error tolerance specification with high probability, and an upper bound for the distance between the sketched and the high-dimensional solutions. Our analysis shows that the projection not only reduces the dimension of the problem but also regularises the reduced system against sketching error. Our numerical simulations suggest speed improvements of two orders of magnitude in exchange for a small loss in the accuracy of the prediction.
\end{abstract}

\maketitle
\tableofcontents

\section{Introduction}
\label{Sec:Introduction}
Motivated by applications in digital manufacturing twins and real-time simulation in robotics, 
we consider the implementation of the Finite Element Method (FEM) in high-dimensional discrete models associated with elliptic partial differential equations (PDE). In particular, we focus on the many-query context, where a stream of approximate solutions are sought for various PDE parameter fields \cite{ElmanSilvesterWathen}, aiming to expedite computations in situations where speedy model prediction is  critical. Realising real-time simulation with high-dimensional models is instrumental to enable digital economy functions and has been driving developments in model reduction over the last decade \cite{Hartmann2018}. Reducing the computational complexity of models is also central to the practical performance of statistical inference and  uncertainty quantification algorithms, where a multitude of model evaluations are necessary to achieve convergence  \cite{LordPowellShardlow}. When real-time prediction is coupled with noisy sensor data, as in the digital twins paradigm, a fast, somewhat inaccurate model prediction typically suffices \cite{Calvetti}.

Our approach is thus tailored to applications where some of the accuracy of the solution can be traded off with speed. In these circumstances the framework of randomised linear algebra presents a competitive alternative \cite{Woodruff}. In the seminal work \cite{DrineasMahoneyER}, Drineas and Mahoney propose an algorithm for computing the solution of the Laplacian of a graph, making the case for sampling the rows of the matrices involved based on their statistical leverage scores. Despite aimed explicitly for symmetric diagonally dominant systems arising, their approach provides inspiration for the numerical solution of PDEs on unstructured meshes. Apart from the algebraic resemblance to the Galerkin FEM systems, the authors introduced sampling based on leverage scores of matrices through the concept of `effective resistance' of a graph derived by mimicking Ohmic relations in resistor networks. As it turns out the complexity of computing the leverage scores is similar to that of solving the high-dimensional problem deterministically, however efficient methods to approximate them have since been suggested \cite{DMMW}. More recently, Avron and Toledo have proposed an extension of \cite{DrineasMahoneyER} for preconditioning the FEM equations by introducing the `effective stiffness' of an element in a finite element mesh \cite{AvronToledo}. Specifically, for sparse symmetric positive definite (SSPD) stiffness matrices, they derive an expression for the effective stiffness of an element and show its equivalence to the statistical leverage scores. Sampling $O(n\log n)$ elements leads to a sparser preconditioner.

In situations where a single, high-dimensional linear system is sought, randomised algorithms suited to SSPD systems are readily available. The methods of Gower and Richtarik for example randomises the row-action iterative methods by taking a sequence of random projections onto convex sets \cite{GowerRichtarik}. This algorithm is equivalent to a stochastic gradient descent method with provable convergence, while their alternative approach in \cite{GowerRichtarik2} iteratively sketches the inverse of the matrix. In \cite{BertsekasYu09}, Bertsekas and Yu present a Monte Carlo method for simulating approximate solutions to linear fixed-point equations, arising in evaluating the cost of stationary policies in Markovian decisions. Their algorithm is based on approximate dynamic programming and has subsequently led to \cite{PolydoridesWangBertsekas}, that extends some of the proposed importance sampling ideas in the context of linear ill-posed inverse problems. 

Real-time FEM computing at the many query paradigm, is hindered by two fundamental challenges: the fast assembly of the stiffness matrix for each parameter field, and the efficient solution of the resulting system to the required accuracy. To mitigate these, is to compromise slightly on the accuracy in order to capitalise on speed. To achieve this we first transform the linear SSPD system into an overdetermined least squares problem, and then project its solution this onto a low-dimensional subspace. This mounts to inverting a low-dimensional, dense matrix whose entries are perturbed by random errors. Our emphasis and contributions are in developing the projected sketching algorithm, and in optimising the sampling process so that it is both efficient in the multi-query context and effective in suppressing the  variance of the solution. We also analyse the complexity of our algorithm and derive, probabilistic error bounds for quality of the approximation. 

Our paper is organised as follows: In section \ref{Sec:fem} we provide a concise introduction to the Galerkin formulation for elliptic boundary value problems, and subsequently derive the projected least squares formulation of the problem. We then describe the sampling distribution used in the sketching and provide the conditions under which the reduced sketched system has a unique solution. Section \ref{sec:algorithm} contains a description of our algorithm, and our main result that describes the complexity of our algorithm in achieving an error tolerance in high probability. We then provide an error analysis addressing the various types of errors imparted on the solution through the various stages of the methodology, before concluding with some numerical experiments based on the steady-state diffusion equation.  

\subsection{Notation}
Let $[m]$ denote the set of integers between 1 and $m$ inclusive. For a matrix $X \in \mathbb{R}^{m \times n}$, $X_{(\ell)}$ and $X^{(\ell)}$ denote its $\ell$-th row and column respectively, and $X_{i j}$ its $(i,j)$-th entry. $X^\dag$ is the pseudo-inverse of $X$ and $\kappa(X)$ its condition number. If $m \geq n$ we define the singular value decomposition $X = U_X \Sigma_X V_X^T$ where $U_X \in \mathbb{R}^{m \times n}$, $\Sigma_X \in \mathbb{R}^{n \times n}$ and $V_X \in \mathbb{R}^{n \times n}$. Unless stated otherwise, singular values and eigenvalues are ordered in non-increasing order. Analogously, for a symmetric and positive definite matrix $A \in \mathbb{R}^{m \times m}$, $\lambda_{\max(A)}$ is the largest eigenvalue, and $\lambda_{\min(A)}$ the smallest. By $\mathrm{nnz}(A)$ we denote the number of non-zero elements in $A$. Further we write $\|\cdot\|$ for the Euclidean norm for a vector or the spectral norm of a matrix and $\| \cdot \|_F$ the Frobenius norm of a matrix. For matrices $X$ and $Y$ with the same number of rows $(X|Y)$ is the augmented matrix formed by column concatenation. The identity matrix is expressed as $I$ or $I_n$ to specify its dimension $n$ when important to the context. We write $y \otimes 1_n$ for the Kronecker product of vector $y$ with the ones vector in $n$ dimensions. 

\section{Galerkin finite element method preliminaries}\label{Sec:fem}

Consider the elliptic partial differential equation 
\begin{equation}\label{pde}
 - \nabla \cdot p \nabla u = f \quad \mathrm{in} \quad \Omega,
\end{equation}
on a bounded, simply connected domain $\Omega \subset \mathbb{R}^d$, $d \in \{2,3\}$ with Dirichlet conditions 
\be\label{bnd}
u = g^{(D)} \quad \mathrm{on} \quad \partial\Omega, 
\ee
on a Lipschitz smooth boundary $\partial \Omega$. Further let $p$ a bounded positive parameter function in the Banach space $L^\infty(\Omega)$ such that
\be\label{padmi}
0 < p_{\min} \leq p \leq p_{\max} < \infty \quad  \mathrm{on} \quad \Omega \cup \partial \Omega,
\ee
for some finite constants $p_{\min}$ and $p_{\max}$. Multiplying \eqref{pde} by an appropriate test function $v$, then integrating over the domain and invoking the divergence theorem yields
\be\label{variational}
\int_\Omega \mathrm{d}x \, \nabla u \cdot p \nabla v = \int_\Omega \mathrm{d}x \, f v,
\ee 
where $\mathrm{d}x$ denotes the $d$-dimensional integration element. Using the standard definition of the Sobolev space on this domain as \be
\mathcal{H}^1(\Omega) \doteq \Bigl \{ u \in L^2(\Omega) \Bigl |  \frac{\partial u}{\partial x_q}\in L^2(\Omega), \quad q=1,\ldots,d  \Bigr \},
\ee
where $L^2(\Omega)$ is the space of square-integrable functions on $\Omega$ we define the solution and test function spaces respectively by
\be
\mathcal{H}^1_U \doteq \Bigl \{ u \in \mathcal{H}^1(\Omega) \Bigl | u = g^{(D)} \; \mathrm{on}\; \partial \Omega \Bigr \},
\quad
\mathcal{H}^1_0 \doteq \Bigl \{ v \in \mathcal{H}^1(\Omega) \Bigl | v = 0 \; \mathrm{on}\; \partial \Omega \Bigr \}.
\ee
Let $f \in L^2(\Omega)$ and $g^{(D)} \in \mathcal{H}^{1/2}(\partial \Omega)$, where the Sobolev space $\mathcal{H}^{1/2}$ is to be understood in terms of a surjective trace operator from $\mathcal{H}_U^1(\Omega)$ to $\mathcal{H}^{1/2}(\partial \Omega)$. Then the weak form of the boundary value problem \eqref{pde}-\eqref{bnd} is to find a function $u \in \mathcal{H}^1_U$ such that 
\be\label{weak}
\int_\Omega \mathrm{d}x \, \nabla u \cdot p \nabla v = \int_\Omega \mathrm{d}x \, f v, \qquad \forall v \in \mathcal{H}^1_0.
\ee
The existence and uniqueness of the weak solution $u$ is guaranteed by the Lax-Milgram theorem \cite{ElmanSilvesterWathen}.

To derive the Galerkin finite element approximation method from the weak form \eqref{weak}, we consider $\mathcal{T}_\Omega \doteq \{\Omega_1, \ldots, \Omega_k\}$ a mesh comprising $k$ elements, having $n$ interior and $n_\partial$ boundary vertices (nodes). Further let $\mathcal{S}^1_\Omega \subset \mathcal{H}^1_0$ the conforming finite dimensional space associated with the chosen finite element basis defined on $\mathcal{T}_\Omega$. Choosing 
$$
\mathcal{S}^1_\Omega \doteq \mathrm{span}\{\phi_1,\ldots, \phi_{n}, \ldots, \phi_{n+n_\partial} \}
$$
to comprise linear interpolation shape functions with local support over the elements in $\mathcal{T}_\Omega$ then we can express the FEM approximation of $u$ in this basis for a set of coefficients $u_1,\ldots,u_{n+n_\partial}$ as   
\be\label{uexp}
u = \sum_{i=1}^{n} u_i \phi_i + \sum_{i=n+1}^{n+n_\partial} \ u_i \phi_i.
\ee
We have made slight abuse of notation by using $u$ for the function in $\mathcal{H}^1_U$ as well as its FEM approximation in $\mathcal{S}^1_\Omega$. In effect, the finite element formulation of the boundary value problem is to find $u \in \mathcal{S}^1_\Omega$ such that
\be
\sum_{\Omega_\ell \in \mathcal{T}_\Omega} \int_{\Omega_\ell} \mathrm{d}x\, \nabla u \cdot p \nabla v = \sum_{\Omega_\ell \in \mathcal{T}_\Omega} \int_{\Omega_\ell} \mathrm{d}x\, f v, \quad \forall v \in \mathcal{S}^1_\Omega.
\ee 
We further define the element-average coefficients 
\be
p_\ell = \frac{1}{|\Omega_\ell|}\int_{\Omega_\ell} \mathrm{d}x\, p , \quad \mathrm{and} \quad f_\ell = \frac{1}{|\Omega_\ell|}\int_{\Omega_\ell} \mathrm{d}x\, f, \quad \ell=1,\ldots, k
\ee
and applying the Dirichlet boundary conditions on the boundary nodes $n_\partial$ we arrive at the Galerkin system of equations for the vector $\{u_1,\ldots, u_n\}$
\be\label{gal}
\sum_{j=1}^{n} \Bigl ( \sum_{\Omega_\ell \in \mathcal{T}_\Omega} \int_{\Omega_\ell} \mathrm{d}x\, \nabla \phi_i \cdot p_\ell \nabla \phi_j \Bigr ) u_j = \sum_{\Omega_\ell \in \mathcal{T}_\Omega} \int_{\Omega_\ell} \mathrm{d}x\, f_\ell \phi_i, \quad i=1,\ldots, n.
\ee
The equations in \eqref{gal} are expressed in a matrix form as 
\be\label{aub}
A u = b,
\ee
where $A \in \mathbb{R}^{n \times n}$ is the symmetric, sparse and positive-definite stiffness matrix, whose dependence on the parameters $p$ is implicit and suppressed for clarity. The FEM construction guarantees the consistency of the system \eqref{aub}, thus $b \in \mathbb{R}^{n}$ is always in the column space of $A$ and consequently it admits a unique solution $u_{\mathrm{opt}} = A^{-1}b$. As we focus to the efficient approximation of $u_{\mathrm{opt}}$ in the many query context we content with two challenges: the efficient assembly of the stiffness matrix, and the speedy solution of the resulted FEM system. 

\subsection{The stiffness matrix}\label{Sec:assfem}

Let $\mathcal{I}_\ell$ is the index set of the $d+1$ vertices of the $\ell$th element, and consider $D_\ell \in \mathbb{R}^{d \times n}$ to be the sparse matrix holding the gradients of the linear shape functions $\phi_i$ where $i \in \mathcal{I}_\ell$. In this $D_\ell^{(i)}$ is then a constant gradients vector associated with the $i$th node of $\Omega_\ell$, and let $z_\ell = |\Omega_\ell|p_\ell$ the element of a vector $z \in \mathbb{R}^k$ such that $Z^2 = \mathrm{diag}(z \otimes 1_d)$ and $D \in \mathbb{R}^{kd \times n}$ a row concatenation of $D_\ell$ matrices for all elements. 
If we define as $Y_\ell = \sqrt{z_\ell}D_\ell$ and $Y \in \mathbb{R}^{kd \times n}$ the concatenation of the $Y_\ell$ matrices as 
\be\label{Ydef}
Y = Z D
\ee
then the stiffness matrix takes the form of a high-dimensional sum or product of sparse matrices 
\be \label{Adec}
A = \sum_{\ell=1}^k Y_\ell^T Y_\ell = Y^T Y,
\ee
which for large $k$ require efficient assembly using reference elements and geometry mappings \cite{KirbyLogg}. The above construction typically leads to a stiffness matrix that is well-conditioned for inversion with the exception of acute element skewness \cite{Higham} and parameter vectors with wild variation \cite{Vavasis}, which cause the the condition number $\kappa(A)$ to increase dramatically. Explicit bounds on the largest and smallest eigenvalues of $A$, and respectively the singular values of $Y$, are given in \cite{KamenskiHuangXu}.   

\section{A regularised sketched formulation}

The sought solution $u_{\mathrm{opt}} = A^{-1}b$ can be alternatively obtained by solving the over-determined least squares problem
\be\label{uLS}
u_{\mathrm{opt}} = u_{\mathrm{LS}} = \arg\min_{u \in \mathbb{R}^n} \|Y u - (Y^T)^\dag b\|^2,
\ee
since 
$$
u_{\mathrm{LS}} = (Y^TY)^{-1} Y^T (Y^T)^\dag b = A^{-1}Y^T (Y^T)^\dag b = A^{-1}b = u_{\mathrm{opt}}.
$$
The fact that the above problem is over-determined implies, at least to some extent, robustness against noise, such as random perturbations on the elements of the matrix $Y$ and vector $b$. A similar error is induced by randomised sketching where we replace \eqref{uLS} with
\be\label{SKuLS}
 {\hat u}_{\mathrm{LS}} = \arg\min_{u \in \mathbb{R}^n} \|\hat Y u - (\hat Y^T)^\dag b\|^2,
\ee
and look for a random approximation $\hat{Y}$ of $Y$ in the sense that ${\hat u}_{\mathrm{LS}} \approx u_{\mathrm{LS}}$. We note that $\hat{Y}$ and $Y$ don't have to be similar as such, e.g. have the same dimensions, as long as the problems are well defined and the optimisers remain similar.
Following \cite{DrineasMahoneyER} and \cite{Pilanci} we seek to approximate $Y$ with some sketch $\hat Y$ by sampling and scaling rows according to probabilities that will be specified later. The number of rows in $\hat{Y}$ in that case equals the number of drawn samples.
Clearly $\hat{Y}$ must have at least $n$ rows as otherwise the problem \eqref{SKuLS} will be under-determined and, due to the non-uniqueness of the solution, the error could become arbitrarily large. On the other hand, if around $n \log(n)$ rows are sampled from a suitable distribution, then Drineas and Mahoney show that the resulting sketch is a good approximation with high probability. However, if substantially less than $n \log(n)$ samples are drawn then the sketching induced error outweighs its computational benefits. In order to understand how this issue can be addressed we note that, if $\hat Y$ has full column-rank and thus the optimiser of \eqref{SKuLS} is unique, the solution of the sketched problem can be obtained by solving the linear system
$$
 \hat Y^T \hat Y u = b,
$$
which is equivalent to solving 
\be\label{equiverr}
 Y^T Y u = b + (Y^T  Y(\hat Y^T \hat Y)^{-1} - I)b = \hat b.
\ee
From \eqref{equiverr} it becomes clear that the sketching induced error can be regarded as an error on the right-hand side of the linear system \eqref{aub} or the least squares problem \eqref{uLS}. We can easily obtain a bound for the relative error given by
$$
\frac{\|\hat b - b\|}{\|b\|} \leq \|Y^T  Y(\hat Y^T \hat Y)^{-1} - I\|
$$
A standard way of dealing with noise as in \eqref{equiverr} is regularisation \cite{regularisationLinSystem}. Suppose that there exists a low-dimensional subspace 
\be
\mathcal{S}_\rho \doteq \{\Psi r\, |\, r \in \mathbb{R}^\rho\},
\ee
spanned by a basis of $\rho \ll n$ orthonormal functions arranged in the columns of matrix $\Psi$, and assume that is sufficient to approximate $u_{\mathrm{opt}}$ within some acceptable level of accuracy, in the sense of incurring a small subspace error $\|(I-\Pi)u_{\mathrm{opt}}\|$. The orthogonal projection operator $\Pi\dot =\Psi \Psi^T$ maps vectors from $\mathbb{R}^n$ onto the subspace $\mathcal{S}_\rho$. Of course, such a subspace can't accommodate all but rather only sufficiently regular $u \in \mathbb{R}^n$. For that reason $\mathcal{S}_\rho$ has to be constructed using prior information (e.g. degree of smoothness) about the solution. Orthogonality of $\Psi$ ensures for any $u_{\mathrm{opt}} = \Pi u_{\mathrm{opt}} + (I - \Pi) u_{\mathrm{opt}}$ the existence of a unique, optimal low-dimensional vector $r_{\mathrm{opt}}$ satisfying 
\be
\Psi r_{\mathrm{opt}} = \Pi u_{\mathrm{opt}}.
\ee 
In these conditions we can pose a projected-regularised least-squares problem replacing \eqref{uLS} by
\be\label{Piu}
\Pi u_{\mathrm{opt}} \approx u_{\mathrm{reg}} = \arg\min_{u \in \mathcal{S}_\rho} \|Y u - (Y^T)^\dag b\|^2,
\ee
in order to improve the robustness of the solution against sketching-induced errors. The problem in \eqref{Piu} still involves high-dimensional quantities such as $Y$ and $b$, but the solution is unique as soon as $\mathcal{S}_\rho$ and the null-space of $Y$ have $\{0\}$ intersection. We start by introducing the low dimensional problem\footnote{We emphasise the contrast between the projected equations in \eqref{eq:lowdim} and the projected variable least squares problem
$$
r' = \arg\min_{r \in \mathbb{R}^\rho} \bigl \|A \Psi r - b \bigr\|^2,
$$
whose solution is
\begin{align*}\nonumber
r' =(\Psi^TA^2\Psi )^{-1}\Psi^TA b =\Psi^T u+(\Psi^TA^2\Psi )^{-1}\Psi^TA^2(I-\Pi) u,
\end{align*}
and incurs a subspace regression error term that is quadratic in $A$. Moreover, note that the right hand side vector in the normal equations $\Psi^TA^TA \Psi r' = \Psi^T A^T b$ has dependence on the parameter through $A$.}
    \begin{equation}\label{eq:lowdim}
        r_{\mathrm{reg}} = \arg\min_{r \in \mathbb{R}^\rho} \|Y \Psi r - (Y^T)^\dag b\|^2.
    \end{equation}
A solution $r_{\mathrm{reg}}$ of \eqref{eq:lowdim} yields a solution $u_{\mathrm{reg}} = \Psi r_{\mathrm{reg}}$ of \eqref{Piu} because the columns of $\Psi$ form an ONB of $\mathcal{S}_\rho$. In addition, we have the following.
\begin{lemma}\label{lem:projection_transform}
    If $Y$ has full column rank and the columns of $\Psi$ form an ONB of $\mathcal{S}_\rho$ so that $\Pi = \Psi \Psi^T$ is the projection onto $\mathcal{S}_\rho$, then
    \begin{equation}\label{eq:eqivprob}
        \arg\min_{u \in \mathcal{S}_\rho} \|Y u - (Y^T)^\dag b\|^2 = \arg\min_{u \in \mathcal{S}_\rho} \|Y\Pi u - (\Psi^T Y^T)^\dag \Psi^T b\|^2.
    \end{equation}
    In particular, both problems have a unique solution.
    \begin{proof}
        Both problems have unique solutions because $\mathcal{S}_\rho$ is convex and $Y$ has (by assumption) full column rank. Therefore it suffices to show that there exists an element $u_{\mathrm{reg}} \in \mathcal{S}_\rho$ that solves both problems. The solution $r_{\mathrm{reg}}$ of \eqref{eq:lowdim} can be found explicitly by solving the linear system
        \begin{equation*}
            \Psi^T Y^T Y \Psi r = \Psi^T Y^T (Y^T)^\dag b \iff  r_{\mathrm{reg}} = (\Psi^T Y^T Y \Psi)^{-1}\Psi^T b.
        \end{equation*}
        We have used that $Y$ has full column rank so that $Y^T (Y^T)^\dag = I$ and $\Psi^T Y^T Y \Psi$ is invertible. Similarly we may consider 
        \begin{equation*}
            \arg\min_{r \in \mathbb{R}^\rho} \|Y \Pi \Psi r - (\Psi^T Y^T)^\dag \Psi^T b\|^2,
        \end{equation*}
        which produces solutions $r_\Psi$ such that $\Psi r_\Psi$ is a solution of the right-hand side of \eqref{eq:eqivprob}. Since $\Pi \Psi = \Psi$ and $Y \Psi$ has full column rank we can write $r_\Psi$ as
        \begin{equation*}
            \Psi^T Y^T Y \Psi r_{\Psi} = \Psi^T Y^T(\Psi^T Y^T)^\dag \Psi^T b  \iff  r_\Psi = (\Psi^T Y^T Y \Psi)^{-1}\Psi^T b.
        \end{equation*}
        We conclude that $\Psi (\Psi^T Y^T Y \Psi)^{-1}\Psi^T b$ is a solution to both sides of \eqref{eq:eqivprob} which completes the proof.
    \end{proof}
\end{lemma}
The right hand side of \eqref{eq:eqivprob} has a very natural interpretation and is obtained by embedding the rows of $Y$, the vector $b$ and the variable $u$ in $\mathcal{S}_\rho$ using its low dimensional representation from the basis induced by the columns of $\Psi$. In view of Lemma \ref{lem:projection_transform} we may regularise the problem from \eqref{SKuLS} and obtain an embedded sketched counterpart to \eqref{Piu} as 
\be\label{SPiu}
{\hat u}_{\mathrm{reg}} = \arg\min_{u \in \mathcal{S}_\rho} \|\hat Y \Pi u - (\Psi^T \hat Y^T)^\dag \Psi^T b\|^2.
\ee
We argue that \eqref{SPiu} is much more robust to the noise imparted by the approximation $\hat Y$ and produces solutions with controlled errors even if substantially less than $n$ suitably drawn samples are used for the approximation. In order to see why, notice that the problem \eqref{SPiu} can be expressed in terms of the low-dimensional vector of coefficients
\be\label{eqn:requiv}
\hat{r}_{\mathrm{reg}} = \arg\min_{r \in \mathbb{R}^\rho} \|\hat Y \Psi r - (\Psi^T \hat Y^T)^\dag \Psi^T b\|^2.
\ee
so that $\Psi \hat{r}_{\mathrm{reg}} = {\hat u}_{\mathrm{reg}}$. Recalling that $A=Y^TY$, it is convenient to introduce
\be
X = Y \Psi \quad \mathrm{and} \quad G = X^TX = \Psi^T A \Psi,
\ee
together with their sketched approximations
\be
\hat X = \hat Y \Psi \quad \mathrm{and} \quad \hat G = \hat X^T \hat X.
\ee
\begin{lemma}\label{lem:uReg}
If $ \hat X = \hat Y \Psi $ has full column rank then the solution of the least-squares problem \eqref{eqn:requiv} is given by $\hat{r}_{\mathrm{reg}} = \hat G^{-1} \Psi^T b$ and we have
\be
{\hat u}_{\mathrm{reg}} = \Psi \hat{r}_{\mathrm{reg}} = u_{\mathrm{reg}} + \Psi (\hat G^{-1} G - I)\Psi^T u_{\mathrm{reg}} .
\ee
where $u_{\mathrm{reg}}$ and ${\hat u}_{\mathrm{reg}}$ are the solutions of \eqref{Piu} and \eqref{SPiu} respectively.
\end{lemma}
\begin{proof} 
If $\hat Y\Psi$ has linearly independent columns then $\Psi^T \hat{Y}^T ( \Psi^T \hat{Y}^T)^\dagger = I$ and the solution $\hat{r}_{\mathrm{reg}}$ of \eqref{eqn:requiv} solves
$$
\hat{G} r = \Psi^T b.
$$
Again $\hat G$ is invertible because $\hat{Y} \Psi$ has linearly independent columns and the first claim follows. The matrix $A$ is positive definite which implies that $G$ is positive definite and $u_{\mathrm{reg}}= \Psi G^{-1}\Psi^T b$. The matrix $\Psi$ has orthonormal columns which implies $\Psi^T b = G \Psi^T u_{\mathrm{reg}}$. Since ${\hat u}_{\mathrm{reg}} = \Psi \hat{r}_{\mathrm{reg}}$ we can use the formula we have just shown and obtain
\begin{align*}
{\hat u}_{\mathrm{reg}} &= \Psi \hat{r}_{\mathrm{reg}} \\
& = \Psi \hat G^{-1} \Psi^T b \\
& = \Psi \hat G^{-1} G \Psi^T u_{\mathrm{reg}}\\
& = \Psi \hat G^{-1} (\hat G + (G - \hat G)) \Psi^T u_{\mathrm{reg}}\\
&= \Pi u_{\mathrm{reg}} + \Psi (\hat G^{-1} G - I)\Psi^T u_{\mathrm{reg}} \\
&= u_{\mathrm{reg}} + \Psi (\hat G^{-1} G - I)\Psi^T u_{\mathrm{reg}}
\end{align*}
where the last identity is due to $u_{\mathrm{reg}} \in \mathcal{S}_\rho $.
\end{proof}
In order to understand the effect of row sampling and why it can be a good approximation, we start by writing
\be \label{Gdec}
G = \sum_{j=1}^{kd} X_{(j)}^T X_{(j)} = X^T X \quad \mathrm{and} \quad A = \sum_{j=1}^{kd} Y_{(j)}^T Y_{(j)} = Y^T Y
\ee
as a sum of outer products of rows. Introduce for some sample size $c \in \mathbb{N}$ the iid random indices $\mathbf{i}_1, \dots, \mathbf{i}_c$ taking values in $[kd]$ with distribution
\begin{equation}
    \mathbb{P}(\mathbf{i}_j = i) = q_i
\end{equation}
for each $j \in [c]$ and $i \in [kd]$. Instead of \eqref{Gdec} we may consider the sketch
\begin{equation}\label{eq:sketchsum}
    \hat{G} = \frac{1}{c} \sum_{j=1}^{c} \frac{1}{q_{\mathbf{i}_j}}X_{(\mathbf{i}_j)}^T X_{(\mathbf{i}_j)}.
\end{equation}
If we define the random matrix $R \in \mathbb{R}^{kd \times c}$ and the random diagonal matrix $W \in \mathbb{R}^{c \times c}$ via
\begin{equation}\label{eq:RW}
    R_{ij} = \begin{cases} 1 & \mathrm{if} \quad \mathbf{i}_j = i \\
    0 & \mathrm{if} \quad \mathbf{i}_j \neq i \end{cases}, \qquad W_{jj} = \frac{1}{\sqrt{c q_{\mathbf{i}_j}}},
\end{equation}
then can put $S = RW$ and construct the sketch $\hat{G}$ as
\begin{equation}
    \hat{G} = X^T SS^T X = X^T R W^2 R^T X.
\end{equation}
Lastly, we can write $\hat{Y} = S^T Y$ as well as $\hat{X} = \hat{Y}\Psi = S^T Y \Psi$ for the sketches of $Y$ and $X$. A simple computation together with an application of the strong law of large numbers shows the following.
\begin{proposition}[Lemma 3 and 4 in \cite{DrineasMahoneyKannan}]\label{prop:asymptotics}
    Assume that the sampling probabilities satisfy the consistency condition
    \begin{equation}\label{eq:SamplingConsistency}
        X_{(j)} \neq 0 \implies q_j > 0 \qquad \forall j = 1, \dots , kd.
    \end{equation} 
    In this case we have for the matrix $\hat{G}$ as defined in \eqref{eq:sketchsum} that $\mathbb{E}[\hat{G}] = G$ and $\mathbb{E}[\|\hat{G}-G\|_F^2] = \mathcal{O}\left(c^{-1}\right)$. As a consequence, $\hat{G} \to G$ almost surely for $c \to \infty$.
\end{proposition}
Proposition \ref{prop:asymptotics} summarises the asymptotic properties of the used sketch. The condition \eqref{eq:SamplingConsistency} is very mild and holds for a wide range of distributions such as sampling from scaled row norms or uniform sampling. The convergence rate of $c^{-1}$ cannot be improved although the constant depends on the chosen probabilities $q_j$. In other words, as long as we sample all non-zero rows with positive probability we will obtain a sketch that has good asymptotic properties when considered as an approximation for $G$. However, in order to find good sampling probabilities $q_j$ we have to consider the non-asymptotic behaviour of the sketch. In fact, the main purpose of the regularisation/dimensionality reduction was to avoid situations where sampling a large number of rows is necessary. If $\rho \ll n$, then the regularised problem \eqref{eq:lowdim} has substantially fewer degrees of freedom than the high dimensional formulation in \eqref{uLS}. Consequently, the dependence of $G$ on the rows of $X$ is a lot smoother than the dependence of $A$ on $Y_{(j)}$. In other words, approximating $X$ by row sampling has a much smaller effect on the regularised solution $u_{\mathrm{reg}}$ than an approximation of $Y$ with the same sample size $c$ would have on the solution $u$ of the full system \eqref{aub}. For example, a much smaller number of rows needs to be sampled to obtain the correct null-space which results in a full-rank approximation of $G$. Note that, conditional on $\hat{G}$ being invertible, $u_{\mathrm{reg}} \in \mathcal{S}_\rho$ in combination with Lemma \ref{lem:uReg} implies
\begin{equation}\label{eq:error_general}
    \frac{\|u_{\mathrm{reg}} -{\hat u}_{\mathrm{reg}}\|}{\|u_{\mathrm{reg}}\|} \leq \|\hat G^{-1} G - I\|,
\end{equation}
so the randomisation error of the regularised problem is entirely controlled by low dimensional structures. This property is the key to a small sketching error and thus to an overall accurate approximation when only few samples are drawn. Using the notation from before and letting $X = U_X \Sigma_X V_X^T$ be the singular value decomposition of $X$, we can write the bound from \eqref{eq:error_general} as
\begin{equation*}
    \|\hat G^{-1} G - I\| = \|\Sigma^{-1}_X (U^T_X S S^T U_X)^{-1}\Sigma_X - I\|.
\end{equation*}
From the above formulation it becomes apparent that the error will be small if the sketch is constructed such that $(U_X S S^T U_X)^{-1} \approx I$ in spectral norm. We argue that this is essentially equivalent to $U_X S S^T U_X \approx I$. Indeed, we have the following.
\begin{lemma}
    If $\|U^T_X S S^T U_X - I\| < \varepsilon < 1$ then
\begin{equation*}
    1-\varepsilon \leq \frac{\| U^T_X S S^T U_X - I\|}{\|(U^T_X S S^T U_X)^{-1} - I \|} \leq 1+\varepsilon.
\end{equation*}
    \begin{proof}
        Under the condition of the lemma we know that $U_X S S^T U_X$ is invertible and that
        \begin{equation*}
            \|U^T_X S S^T U_X\| \leq \|I\| + \|U^T_X S S^T U_X - I\| = 1 + \varepsilon
        \end{equation*}
        which implies the upper bound by considering the estimate
        \begin{align*}
            \|U^T_X S S^T U_X - I\| &\leq \|U^T_X S S^T U_X\| \|(U^T_X S S^T U_X)^{-1} - I\| \\
            &\leq (1 + \varepsilon)\|(U^T_X S S^T U_X)^{-1} - I\|.
        \end{align*}
        Denote by $\lambda_i(U_X S S^T U_X)$ the $i$-th eigenvalue of $U_X S S^T U_X$. Then we may write
        \begin{align*}
            \|(U_X S S^T U_X)^{-1} - I\| &= \max_{i=1}^{\rho} \lvert 1- \lambda_i^{-1}(U^T_X S S^T U_X)\rvert \\ &= \max_{i=1}^{\rho} \frac{\lvert 1- \lambda_i(U^T_X S S^T U_X)\rvert}{\lambda_i(U^T_X S S^T U_X)} \\
            & \leq \frac{\| 1-U^T_X S S^T U_X\|}{\lambda_{\min}(U^T_X S S^T U_X)}
        \end{align*}
        where $\lambda_{\min}(U_X S S^T U_X)$ is the smallest eigenvalue. By assumption of the lemma \begin{equation*}
            \lvert 1- \lambda_{\min}(U^T_X S S^T U_X)\rvert \leq \varepsilon\implies \lambda_{\min}(U^T_X S S^T U_X) \geq 1-\varepsilon
        \end{equation*}
        which implies the claim after dividing by $\| 1-U^T_X S S^T U_X\|$ and taking the inverse.
    \end{proof}
\end{lemma}
An approximation of $U^T_X S S^T U_X$ can be obtained by sampling with probabilities that are proportional to the statistical leverage scores
\begin{equation}\label{eq:leverage_scores}
    \ell_i(X) = \ell_i(U_X) = \|(U_X)_{(i)}\|^2,
\end{equation}
i.e. the row norms of the left singular vectors of $X$ \cite{DMMW}. At first sight it seems that taking sampling probabilities proportional to the leverage scores in \eqref{eq:leverage_scores} in order to obtain a sketch of \eqref{eq:lowdim} is very similar to using the leverage scores of $Y$ to obtain \eqref{SKuLS} from \eqref{uLS} as was proposed by Drineas and Mahoney in \cite{DrineasMahoneyER} for a similar problem. A key difference is that $X$ is tall and dense while $Y$ is sparse and thus $G$ is quite different to the initial stiffness matrix $A$. Consequently, an interpretation of the leverage scores from \eqref{eq:leverage_scores} in terms of effective stiffness \cite{AvronToledo} is, to the best of our knowledge, not possible. The following Lemma will be useful for our further developments.
\begin{lemma}[\cite{tropp2015} section 6.4]\label{lem:concentration}
    Assume that $S$ is constructed as before with sampling probabilities $q_i$ satisfying
    \begin{equation}\label{eq:SamplingBound}
        q_i \geq \beta \frac{\ell_i(X)}{\rho} \quad i = 1, \dots, kd
    \end{equation}
    for some $\beta \in (0,1]$. Then we have $\forall \varepsilon > 0$
    \begin{equation}\label{eq:concentration_inequality}
        \mathbb{P}\left(\|U^T_X S S^T U_X - I\| \geq \varepsilon \right) \leq 2 \rho  \exp\left(-\frac{3 c \beta \varepsilon^2 }{12 \rho + 4 \rho \varepsilon}\right)
    \end{equation}
\end{lemma}
An important corollary of the above lemma is that a sketch which is constructed by sampling from leverage score probabilities will virtually always be invertible and therefore the sketched problem \eqref{eqn:requiv} has a unique solution. The following result states that this property is preserved even when the rows are re-weighted, an operation which changes the leverage scores.
\begin{proposition}\label{prop:full_rank}
    Let $\Gamma \in \mathbb{R}^{kd \times kd}$ be a diagonal matrix with positive entries, i.e. $\Gamma_{ii} > 0$ for each $i = 1, \dots, kd$. Assume that the sketching matrix $S$ is constructed with sampling probabilities $q_i = \rho^{-1}\ell_i(X)$. For the scaled sketch $\hat{H} = X^T \Gamma S S^T \Gamma X$ we have
    \begin{equation}
        \mathbb{P}(\hat{H} ~ \mathrm{is ~ invertible}) = \mathbb{P}(\hat{G} ~ \mathrm{is ~ invertible}) \geq 1 - 2 \rho  \exp\left(-\frac{3 c }{16 \rho}\right)
    \end{equation}
    \begin{proof}
    It is sufficient to show that
    \begin{equation*}
         \hat{H} ~ \mathrm{is ~ invertible} \iff \hat{G} ~ \mathrm{is ~ invertible} \iff U^T_X S S^T U_X ~ \mathrm{is ~ invertible}
    \end{equation*}
    because the probability bound follows immediately from
    \begin{equation*}
        \mathbb{P}(U^T_X S S^T U_X ~ \mathrm{is ~ invertible}) \geq 1-\mathbb{P}\left(\|U^T_X S S^T U_X - I\| \geq 1 \right)
    \end{equation*}
    after applying \eqref{eq:concentration_inequality} from Lemma \ref{lem:concentration}. The above matrices are always positive semi-definite and therefore invertibility is equivalent to positive definiteness. For any diagonal matrix $\Gamma$ it holds that $S^T \Gamma = \hat{\Gamma}S^T$ where $\hat{\Gamma}$ is a random diagonal matrix with entries $\hat{\Gamma}_{jj} = \Gamma_{\mathbf{i}_j\mathbf{i}_j}$. Thus for any $x \in \mathbb{R}^{\rho}$ we have
    \begin{equation*}
        x^T \hat{H} x = (\Sigma_X V_X^T x)^T U^T_X S \hat{\Gamma}^2 S^T U_X (\Sigma_X V_X^T x).
    \end{equation*}
    Since $X$ has full column rank we know that $\Sigma_X V_X^T$ corresponds to a change of basis and $\Sigma_X V_X^T x \neq 0$ whenever $x \neq 0$. It follows that $\hat{H}$ is positive definite if and only if $U^T_X S \hat{\Gamma}^2 S^T U_X$ is positive definite. As $\hat{\Gamma}$ is a diagonal such that $\hat{\Gamma}_{jj} > 0$ with probability $1$, the latter is equivalent to $U^T_X S S^T U_X$ being positive definite. The case of $\hat{G}$ is covered by $\Gamma = I$.
    \end{proof}
\end{proposition}
Proposition \ref{prop:full_rank} states that re-scaling of rows doesn't affect the quality of the sketching matrix regarding its invertibility and after sampling $\rho \log(\rho)$ rows the probability of the sketch being singular decays exponentially fast with each additional draw. In practice this makes knowledge of $\ell_i(X)$ valuable because we only need to sample $\rho \log(\rho) + M$ rows for some moderately large $M$ and obtain a sketch that is virtually never singular. On the other hand, we need at least $\rho$ samples so that there is any hope in obtaining a non-singular matrix. The remarkable thing about Proposition \ref{prop:full_rank} is that the failure probability is \emph{independent} of both, the inner dimension $kd$ of the product $X^T X$ as well as the scaling matrix $\Gamma$ and equivalent to the bound which could be obtained by sampling from $\ell_i(\Gamma X)$. This suggests that a sketch which is constructed by drawing samples from $\ell_i(X)$ is not too different compared to sampling from $\ell_i(\Gamma X)$. This intuition is supported by the following result which describes the change in the leverage scores after re-weighting a single row.
\begin{proposition}[\cite{cohen2015} Lemma 5]\label{prop:leverage_dynamics}
    Let $\Gamma^{\langle i \rangle} \in \mathbb{R}^{kd \times kd}$ be a diagonal matrix with $\Gamma^{\langle i \rangle}_{ii} = \sqrt{\gamma} \in (0,1)$ and $\Gamma^{\langle i \rangle}_{jj} = 1$ for each $j \neq i$. Then 
    \begin{equation}\label{eq:IlevI}
        \ell_{i}(\Gamma^{\langle i \rangle}X) = \frac{\gamma \ell_i(X)}{1 - (1-\gamma)\ell_i(X)} \leq \ell_i(X) 
    \end{equation}
    and for $i \neq j$
    \begin{equation}\label{eq:IlevJ}
        \ell_{j}(\Gamma^{\langle i \rangle}X) = \ell_{j}(X) + \frac{(1-\gamma) \ell^2_{ij}(X)}{1 - (1-\gamma)\ell_i(X)} \geq \ell_j(X)
    \end{equation}
    where $\ell_{ij}(X) = (U_X U_X^T)_{ij}$ are the cross leverage scores.
\end{proposition}
Since $U_X$ has orthogonal columns, we have $\|v\| = \| U_X v\|$ for any $v \in \mathbb{R}^{\rho}$ and thus the cross leverage scores from the above Lemma satisfy
\begin{equation}\label{eq:crossleverageEq}
    \ell_{i}(X) = \sum_{j = 1}^{kd} \ell^2_{ij}(X).
\end{equation} 
For a general diagonal matrix $\Gamma$ as in Proposition \ref{prop:full_rank} we may without loss of generality assume that each entry lies in $(0,1]$ since we can divide the elements by their maximum. The re-weighting can thus be considered as a superposition of single row operations
\begin{equation}
    \Gamma = \prod_{i = 1}^{kd} \Gamma^{\langle i \rangle}
\end{equation}
where the $\Gamma^{\langle i \rangle}$ are as in Proposition \ref{prop:leverage_dynamics}. Since the $\Gamma^{\langle i \rangle}$ commute we can apply them in any order without changing the outcome. Considering Lemma \ref{lem:concentration}, if we could ensure that $\ell_i(X)$ isn't substantially smaller than $\ell_i(\Gamma X)$ then sampling from $q_i = \rho^{-1}\ell_i(X)$ will produce good sketches for $\Gamma X$. 
\paragraph{Large leverage scores $\ell_i(X) \approx 1$} Equation \eqref{eq:IlevI} shows that the relative change of the $i$-th leverage score after a re-weighting of the $i$-th row shrinks when $\ell_i(X) \to 1$. In the extreme case when $\ell_i(X) = 1$ the re-weighting has no effect. In addition to this stability property it trivially holds that $\ell_i(X) \leq 1$ which suggests that large leverage scores are fairly stable when rows are re-weighted.
\paragraph{Small leverage scores $\ell_i(X) \ll 1$} From Equation \eqref{eq:IlevJ} we know that the increase of $\ell_j(X)$ after re-weighting of row $i$ is proportional to $\ell_{ij}(X)$. If the entries of the scaling matrix $\Gamma$ don't vary too much, then \eqref{eq:crossleverageEq} suggests that we can expect the total increase, i.e. after applying $\Gamma^{\langle j \rangle}$ for each $j \neq i$ to be roughly of order $\ell_i(X)-\ell_i^2(X) \approx \ell_i(X)$. On the other hand, small $\ell_i(X)$ are fairly sensitive to re-weighting of row $i$ since $\ell_{i}(\Gamma^{\langle i \rangle} X) \approx (\Gamma^{\langle i \rangle}_{ii})^2\ell_{i}(X)$ in that case. Thus we can expect that the re-weighting of row $i$ will counterbalance the effects from re-weighting the other rows. In addition, we know that
\begin{equation*}
    \sum_{i = 1}^{kd} \ell_i(X) = \sum_{i = 1}^{kd} \ell_i(\Gamma X).
\end{equation*}
Since large leverage scores will likely be quite stable and $\ell_i(\Gamma X) \geq 0$ we would expect that not too many small leverage scores will become large. 

So far we have discussed the projection of the high-dimensional system without providing explicit details on how the basis $\Psi$ is selected. A desired property is to sustain a small projection error for all admissible parameter choices under the constraint $\rho \ll n$. Suitable options include subsets of the right singular vectors of $A$ or orthogonalised Krylov-subspace bases  \cite{Halko}, however these have to be computed for each individual parameter vector which can be detrimental to the speed of the solver. Alternatively, we opt for a generic basis exploiting the smoothness of $u$ on domains with smooth Lipschitz boundaries. A simple choice is to select the basis among the eigenvectors of the discrete Laplacian operator 
\be\label{Delta}
\Delta \doteq D^T Z^2_\Delta D,
\ee
for $Z^2_\Delta = \mathrm{diag}\bigl ([|\Omega_1|,\ldots,|\Omega_k|] \otimes 1_d\bigr )$. From $U_\Delta^T \Delta U_\Delta = \Sigma_\Delta$ and splitting the eigenvectors as
$$
U_\Delta=\bigl ( {U_\Delta}^{(1:n-\rho-1)}| \Psi \bigr ),
$$
such that the columns of $\Psi$ correspond to the last $\rho$ columns of $U_\Delta$, and respectively to the $\rho$ smallest eigenvalues $\{\lambda_{n-\rho-1}(\Delta), \ldots, \lambda_n(\Delta)\}$. In effect, with $\Delta$  constrained by the Dirichlet boundary conditions, the norm $\| \Delta \Psi^{(i)}\|$ provides a measure of the smoothness of $\Psi^{(i)}$ in the interior of $\Omega$. It is not difficult to see that this basis satisfies
$$
\|\Delta \Psi^{(i)}\| \geq \| \Delta \Psi^{(j)}\| \quad \mathrm{for} \quad \rho \geq i>j \geq 1.
$$
We remark that the computation of the basis is computationally very expensive for large $n$, as the eigen-decomposition of $\Delta$ is necessary, however this is only computed once, prior to the beginning of the simulation (offline stage) in an offline stage. After the matrix $\Psi$ has been obtained we can compute the leverage scores $\ell_{i}(Z_\Delta D \Psi)$. The Laplacian $\Delta$ differs from a general stiffness matrix $A$ only by different diagonal weights, i.e. $Z^2_\Delta$ is replaced by the diagonal matrix $Z^2 = Z^2_\Delta \mathrm{diag}\bigl[(p_1,\ldots,p_k) \otimes 1_d\bigr]$ where the $p_i$ contain information about the parameter from \eqref{pde}. Propositions \ref{prop:full_rank} and \ref{prop:leverage_dynamics} along with the developments thereafter suggest that the Laplacian leverage scores $\ell_{i}(Z_\Delta D \Psi)$ can nonetheless be used to construct sketches $\hat{G} = X^T S S^T X$ of the projected matrix $G = X^T X = \Psi^T Y^T Y \Psi$ because the difference in the stiffness matrices is just a diagonal weighting.
\section{Complexity and error analysis}\label{sec:algorithm}
Motivated by the developments from the previous sections we propose the following algorithm for computing solutions to a sequence of $N$ problem of the form \eqref{pde}. We assume that each problem is specified by its parameter vector $z^{(t)} \in \mathbb{R}^{kd}$ for $t = 1, \dots, N$ (see section \ref{Sec:assfem}). 
\begin{algorithm}[ht]
\SetKwInOut{Input}{input}\SetKwInOut{Output}{output}
\Input{Matrices $D \in \mathbb{R}^{kd \times n}$, $\Psi \in \mathbb{R}^{n \times \rho}$, data vector  $\Psi^Tb \in \mathbb{R}^{\rho}$, and sampling probabilities $q_i = \rho^{-1}\ell_{i}(Z_\Delta D\Psi)$ (offline)}
\Output{Parameter dependent solutions $\hat r^{(t)} \in \mathbb{R}^\rho$ where $t = 1, \dots, N$}
  \SetAlgoLined
  Online Simulation\;
  \For{$t\leftarrow 1$ \KwTo $N$}{
  \Input{Parameter vector $z^{(t)} \in \mathbb{R}^{k}$, sample size $c$}
  draw row indices $\mathbf{i}_1, \dots \mathbf{i}_c \stackrel{\mathrm{iid}}{\sim} q$ from $[kd]$\;
  get the sampled indices $J = \bigcup_{j=1}^{c}\{\mathbf{i}_j\}$\; 
  set $c'=\lvert J \rvert$ and write $J = \{\mathbf{j}_1, \ldots, \mathbf{j}_{c'}\}$\;
  compute the frequencies $m_j = \sum_{k=1}^{c} \delta(\mathbf{i}_k=\mathbf{j}_j)$ for $j=1,\ldots,c'$\;
  find $M^2_{jj}=c^{-1} m_{j}q^{-1}_{\mathbf{j}_j}$ for $j=1,\ldots, c'$ and the diagonal matrix $M$\;
  find $\hat{Z}^{2}_{jj} = z^{(t)}_{\mathbf{j}_j}$ for $j=1,\ldots, c'$ and the diagonal matrix $\hat{Z}^2$\;
  assemble the $c' \times \rho$ matrix $\hat{X} = M \hat{Z} D_{(J)} \Psi$\;
  compute reduced system $\hat G = \hat{X}^T \hat{X}$\;
  compute and store $\hat{r}^{(t)} \leftarrow \mathrm{solve}(\hat{G}, \Psi^Tb)$\;
 }
  \caption{Algorithm for simulating the low-dimensional projected solution of the FEM equations for different choices of parameter vectors $p$. Note that as we are sampling with replacement, $c' \leq c$. In the above $\delta (\cdot)$ denotes the indicator function where $\delta (E) = 1$ if the event $E$ has occurred and it is zero otherwise  otherwise. $D_{(J)}$ is the sub-matrix of $D$ whose rows are the (ordered) elements of $J$}.
  \label{alg:sketch1}
\end{algorithm}

The complexity and approximation error of Algorithm \ref{alg:sketch1} are obviously linked. The more samples we draw the better we expect our solutions to be. Although the size of the reduced system matrix $G$ (and therefore its sketched counterpart $\hat{G}$ as well) is independent of $c$, the computational burden for building $\hat{G}$ is higher when drawing more samples. More precisely, we need:
\begin{itemize}
    \item $\mathcal{O}(c)$ operations in order to find $\mathbf{i}_1, \dots \mathbf{i}_c \stackrel{\mathrm{iid}}{\sim} q$. This is possible because $q$ is fixed and we can perform the necessary pre-processing offline \cite{Bringmann2017}.
    \item $\mathcal{O}(c)$ operations for computing the sampled indices $\{\mathbf{j}_1, \dots, \mathbf{j}_{c'}\}$ and their frequencies $m_j$ as this requires a single loop through the set $\{\mathbf{i}_1, \dots, \mathbf{i}_c\}$ of initial samples.
    \item $\mathcal{O}(c')$ operation for assembling the diagonal matrices $M$ and $\hat{Z}$.
    \item $\mathcal{O}(c'\rho)$ operations for computing $M \hat{Z} D_{(J)} \Psi$. This can be achieved since computing $M \hat{Z} D_{(J)}$ requires $\mathrm{nnz}(D_{(J)}) = \mathcal{O}(c')$ multiplications and $\rho \cdot \mathrm{nnz}(M \hat{Z}D_{(J)}) = \rho \cdot \mathrm{nnz}(D_{(J)}) = \mathcal{O}(\rho c')$ multiplications are enough for computing $[M \hat{Z} D_{(J)}] \Psi $ due to sparsity of $D$. 
    \item $\mathcal{O}(c'\rho^2)$ operations in order to build $\hat{G}$ which corresponds to the cost of multiplication for dense matrices.
    \item $\mathcal{O}(\rho^3)$ operations for solving $\hat{G}r = \Psi^Tb$ with a direct method.
\end{itemize}
The sketch $\hat{G}$ will be singular if we draw $c' < \rho$ distinct samples which means that building the sketch $\hat{G}$ dominates the complexity of Algorithm \ref{alg:sketch1}. If the sampling probabilities are a good approximation in the sense that $\beta$ in Lemma \ref{lem:concentration} can be chosen close to $1$, then we need $c = \mathcal{O}(\varepsilon^{-2}\rho \log(\rho))$ samples in order to have a provably controlled error. The worst case, i.e. the the largest increase of $\ell_i(X)$, will be observed if $z^{(t)}_{j} \ll z^{(t)}_{i}$ for $j \neq i$. A parameter $p$ corresponding to such a situation essentially renders the implementation of the classical Galerkin FEM problematic, as $\kappa(A)$ scales to $p_{\max}/p_{\min}$, see Theorem 5.2 in \cite{KamenskiHuangXu}  The following theorem summarises the findings of this section.
\begin{theorem}\label{thm:randerror}
    Let $\varepsilon \in (0,1)$ and $\beta \in (0,1]$ is such that the sampling probabilities $q_i$ from Algorithm \ref{alg:sketch1} satisfy \eqref{eq:SamplingBound}, i.e.
    \begin{equation*}
        q_i \geq \beta \frac{\ell_i(Z D\Psi)}{\rho} \quad i = 1, \dots, kd
    \end{equation*}
    where $Z^2 = \mathrm{diag}(z^{(t)})$. Let $G = X^T X = \Psi^T D^T Z^2 D\Psi$ be the reduced system matrix corresponding to parameter $z^{(t)}$ and $\kappa(G)$ its condition number. For the choice $c = 15\rho \log(15 \rho) \beta^{-1} \varepsilon^{-2} $ Algorithm \ref{alg:sketch1} requires $\mathcal{O}(\rho^3 \log( \rho)\beta^{-1} \varepsilon^{-2} )$ operations and outputs, with probability exceeding $0.999$, a vector $\hat{r}^{(t)}$ that satisfies 
    \begin{equation}
        \frac{\|\hat{r}^{(t)} - G^{-1}\Psi^Tb\|}{\|G^{-1}\Psi^Tb\|} \leq \sqrt{\kappa(G)}\frac{\varepsilon}{1-\varepsilon}.
    \end{equation}
    \begin{proof}
        As stated before, the complexity of Algorithm \ref{alg:sketch1} is $\mathcal{O}(c\rho^2)$ which immediately implies that it requires $\mathcal{O}(\rho^3 \log( \rho)\beta^{-1} \varepsilon^{-2} )$ operations for a single query. It remains to prove the error bound. In view of \eqref{eq:error_general} and the developments thereafter it follows, conditional on $\hat{G}$ being invertible, that
        \begin{align*}
            \frac{\|\hat{r}^{(t)} - G^{-1}\Psi^Tb\|}{\|G^{-1}\Psi^Tb\|} &\leq \|\Sigma^{-1}_{X} (U^T_{X} S S^T U_{X})^{-1}\Sigma_{X} - I\| \\
            & \leq \kappa(X)\|(U^T_{X} S S^T U_{X})^{-1} - I\|\\
            & \leq \kappa(X)\frac{1}{1-\varepsilon}\|U^T_{X} S S^T U_{X} - I\|.
        \end{align*}
        Since $\kappa^2(X) = \kappa(G)$ we only need to show that 
        \begin{equation*}
            \mathbb{P}(\|U^T_{X} S S^T U_{X} - I\| \geq \varepsilon) \leq 0.001
        \end{equation*}
        because $\hat{G}$ is necessarily invertible on that event which implies validity of the estimates from before. But plugging the value for $c$ into \eqref{eq:concentration_inequality} we obtain for any $\rho \geq 1$
        \begin{equation*}
            \mathbb{P}(\|U^T_{X} S S^T U_{X} - I\| \geq \varepsilon) \leq \frac{2}{15}\exp\left(-\frac{29}{16}\log(15\rho)\right) < 0.001.
        \end{equation*}
    \end{proof}
\end{theorem}
Algorithm \ref{alg:sketch1} is most attractive when we can tolerate an error somewhere between 1\% to 10\% in which case we can obtain the solution to a single query in about $\mathcal{O}(\beta^{-1} \rho^3 \log(\rho))$ time. In practice the value for $\beta$ is unobtainable since it requires knowledge of the true leverage scores but considering Lemma \ref{prop:leverage_dynamics} and the arguments thereafter, we expect that for a moderately large $\beta^{-1}$ the required bound will hold for all but a few small leverage scores. The statement in Lemma \ref{lem:concentration} is rather pessimistic when there are few misaligned leverage scores since it requires a uniform bound. For practical purposes we expect that $\beta^{-1}$ can be substituted with a small constant and we take $\varepsilon = 0.1$ which will ensure reglarity of the sketch.
Up until now we have only considered the randomisation error of the sketched solution, i.e. we have analysed $\|\hat{u}_{\mathrm{reg}}-u_{\mathrm{reg}}\|$. However, the the total error of $\hat{u}_{\mathrm{reg}}$ compared to the high dimensional solution $u$ of \eqref{aub} has two components. If we decompose the process into two steps
\begin{align}
    \min_{u \in \mathbb{R}^n} \|Y u - (Y^T)^\dag b\|^2 \quad &\xrightarrow[\|u_{\mathrm{opt}} - u_{\mathrm{reg}}\|]{\makebox[2cm]{\scriptsize Projection}} \quad  \min_{u \in \mathcal{S}_\rho} \|Y u - (Y^T)^\dag b\|^2& \label{eq:projstep}\\[6pt] 
    \min_{u \in \mathcal{S}_\rho} \|Y u - (Y^T)^\dag b\|^2 \quad &\xrightarrow[\|\hat{u}_{\mathrm{reg}}-u_{\mathrm{reg}}\|]{\makebox[2cm]{\scriptsize Sketching}} \quad \min_{u \in \mathcal{S}_\rho} \|\hat{Y} \Pi u - (\Psi^T \hat{Y}^T)^\dag \Psi^T b\|^2, & \label{eq:randstep}
\end{align}
it becomes apparent that even with a perfect sketch, i.e. if we solved the noiseless projected problem \eqref{Piu} and \eqref{eq:randstep} is negligible, we could still not achieve an error smaller than $\|u_{\mathrm{opt}} - \Pi u_{\mathrm{opt}}\|$. The next result tells us that the error from \eqref{eq:projstep} is close to the optimal one.
\begin{theorem}\label{thm:projerror}
Let $u_{\mathrm{opt}}$ be the solution of \eqref{aub} and $u_{\mathrm{reg}}$ be the optimum of \eqref{Piu}. If $\kappa (A)$ is the condition number of the stiffness matrix $A$ and $\Pi = \Psi \Psi^T$ the projection ont $\mathcal{S}_\rho$, then
\begin{equation*}
        \|u_{\mathrm{opt}} - u_{\mathrm{reg}}\| \leq \left(1 + \sqrt{\kappa(A)}\right) \|u_{\mathrm{opt}} - \Pi u_{\mathrm{opt}}\|.
    \end{equation*}
    \begin{proof}
        Recall that $A = Y^T Y$ and $G = X^T X = \Psi^T Y^T Y \Psi$. From the developments in Lemma \ref{lem:uReg} we know that $u_{\mathrm{reg}} = \Psi G^{-1} \Psi^T b$. We may write as before $X = U_X \Sigma_X V_X^T$ so that $G^{-1} = V_X \Sigma_X^{-2} V_X^T$ and
        \begin{align*}
            \|u_{\mathrm{opt}} - u_{\mathrm{reg}}\| &= \|u_{\mathrm{opt}} - \Psi G^{-1} \Psi^T b\| \\
            &= \|u_{\mathrm{opt}} - \Psi G^{-1} \Psi^T Au_{\mathrm{opt}}\| \\
            &= \|u_{\mathrm{opt}} - \Psi G^{-1} \Psi^T A[\Pi + (I - \Pi)]u_{\mathrm{opt}}\| \\
            &\leq \|u_{\mathrm{opt}} - \Psi G^{-1} \Psi^T A\Psi \Psi^T u_{\mathrm{opt}}\| + 
            \|\Psi G^{-1} \Psi^T A(I - \Pi)u_{\mathrm{opt}}\| \\
            &= \|u_{\mathrm{opt}} - \Pi u_{\mathrm{opt}}\| + \left\|\Psi V_X \Sigma_X^{-2} V_X^T (U_X \Sigma_X V_X^T)^T Y (I - \Pi)u_{\mathrm{opt}} \right\| \\
            &\leq  \|u_{\mathrm{opt}} - \Pi u_{\mathrm{opt}}\| \left( 1 + \|\Psi V_X \Sigma_X^{-1} U_X^T Y \|  \right).
        \end{align*}
If we write $\lambda_{\min}(A)$ and $\lambda_{\max}(A)$ for the smallest and largest eigenvalues of $A$, then it must hold that
\begin{equation*}            
\lambda_{\min}(A) \leq  \lambda_{\min}(G) \leq \lambda_{\max}(G) \leq \lambda_{\max}(A) 
        \end{equation*} 
        because $\Psi$ has orthogonal columns. Indeed, if $\mathbb{S}^{n-1}\doteq \{w \in \mathbb{R}^n:\|w\|=1\}$ is the $n$-dimensional unit sphere, then
        \begin{equation*}
            \min_{w \in \mathbb{S}^{n-1}} w^T A w\leq \min_{w \in \mathcal{S}_\rho \cap \mathbb{S}^{n-1}}w^T A w \leq \max_{w \in \mathcal{S}_\rho \cap \mathbb{S}^{n-1}}w^T A w \leq \max_{w \in \mathbb{S}^{n-1}}w^T A w
        \end{equation*}
        is obviously true. Since the columns of $\Psi$ form an ONB of $\mathcal{S}_\rho$ we have
        \begin{gather*}
             \min_{w \in \mathcal{S}_\rho \cap \mathbb{S}^{n-1}}w^T A w =  \min_{w \in \mathbb{S}^{\rho-1}}w^T \Psi^T A \Psi w = \min_{w \in \mathbb{S}^{\rho-1}}w^T G w = \lambda_{\min}(G) \\
             \max_{w \in \mathcal{S}_\rho \cap \mathbb{S}^{n-1}}w^T A w = \max_{w \in \mathbb{S}^{\rho-1}}w^T \Psi^T A \Psi w = \max_{w \in \mathbb{S}^{\rho-1}}w^T G w = \lambda_{\max}(G). 
        \end{gather*}
        Thus, $\|\Sigma^{-1}_X \|^2 =  \lambda^{-1}_{\min}(G) \leq \lambda^{-1}_{\min}(A)$. Clearly we also have $\|Y\|^2 = \lambda_{\max}(A)$. Due to orthogonality we know that $\|\Psi\|=\|V_X\|=\|U_X\|=1$. Combining those estimates we obtain
        \begin{equation*}
            \|\Psi V_X \Sigma_X^{-1} U_X^T Y \| \leq \sqrt{\frac{\lambda_{\max}(A)}{\lambda_{\min}(G)}} \leq \sqrt{\kappa(A)},
        \end{equation*}
        which yields the desired bound.
    \end{proof}
\end{theorem}
If the subspace $\mathcal{S}_{\rho}$ is such that the relative projection error is small, then the norm of $u_{\mathrm{reg}}$ will be similar to the norm of $u_{\mathrm{opt}}$. More precisely,
\begin{equation*}
    \frac{\|u_{\mathrm{reg}}-u_{\mathrm{opt}}\|}{\|u_{\mathrm{opt}}\|} \leq \delta \implies \frac{\|u_{\mathrm{reg}}\|}{\|u_{\mathrm{opt}}\|} \in [1-\delta, 1+\delta]
\end{equation*}
so that Theorem \ref{thm:randerror} applies to $\|u_{\mathrm{reg}}-\hat{u}_{\mathrm{reg}}\|/\|u_{\mathrm{opt}}\|$ with a small $\delta$-dependent constant. By combining the previous two theorems we obtain the following.
\begin{corollary}\label{cor:totalerror}
    Let $\varepsilon_{\mathrm{R}} \in (0,1)$ and assume that the assumptions of Theorem \ref{thm:randerror} are satisfied for $\varepsilon = \varepsilon_{\mathrm{R}}$. If $u_{\mathrm{opt}}$ is the solution of \eqref{aub} and the subspace $\mathcal{S}_{\rho}$ is such that
    \begin{equation*}
       \|u_{\mathrm{opt}} - \Pi u_{\mathrm{opt}}\| \leq \|u_{\mathrm{opt}}\| \varepsilon_{\mathrm{P}}
    \end{equation*}
    for some $\varepsilon_{\mathrm{P}} \in (0,1)$. Then the total error of the solutions $\hat{u}_{\mathrm{reg}} = \Psi \hat{r}$ produced by Algorithm \ref{alg:sketch1} satisfy the bound
    \begin{equation}
        \frac{\|u_{\mathrm{opt}} - \hat{u}_{\mathrm{reg}}\|}{\|u_{\mathrm{opt}}\|} \leq \left(1+\varepsilon_{\mathrm{P}}\sqrt{\kappa(A)}\right)\sqrt{\kappa(G)}\frac{\varepsilon_{\mathrm{R}}}{1-\varepsilon_{\mathrm{R}}} +
         \left(1+\sqrt{\kappa(A)}\right) \varepsilon_{\mathrm{P}}.
    \end{equation}
    \begin{proof}
        We can start with the estimate
        \begin{equation*}
            \frac{\|u_{\mathrm{opt}} - \hat{u}_{\mathrm{reg}}\|}{\|u_{\mathrm{opt}}\|} \leq \frac{\|u_{\mathrm{opt}} - u_{\mathrm{reg}}\|}{\|u_{\mathrm{opt}}\|} + \frac{\|{u}_{\mathrm{reg}} - \hat{u}_{\mathrm{reg}}\|}{\|u_{\mathrm{opt}}\|}.
        \end{equation*}
        Using the estimate from Theorem \ref{thm:projerror} we get
        \begin{equation*}
            \frac{\|u_{\mathrm{opt}} - u_{\mathrm{reg}}\|}{\|u_{\mathrm{opt}}\|} \leq \left(1+\sqrt{\kappa(A)}\right) \frac{\|u_{\mathrm{opt}} - \Pi u_{\mathrm{opt}}\|}{\|u_{\mathrm{opt}}\|} \leq \left(1+\sqrt{\kappa(A)}\right) \varepsilon_{\mathrm{P}}.
        \end{equation*}
        It remains to bound the other term. Since $\Psi$ has orthogonal columns we obtain from Theorem \ref{thm:randerror} 
        \begin{equation*}
            \frac{\|{u}_{\mathrm{reg}} - \hat{u}_{\mathrm{reg}}\|}{\|u_{\mathrm{reg}}\|} \leq  \sqrt{\kappa(G)}\frac{\varepsilon_{\mathrm{R}}}{1-\varepsilon_{\mathrm{R}}} \implies \frac{\|{u}_{\mathrm{reg}} - \hat{u}_{\mathrm{reg}}\|}{\|u_{\mathrm{opt}}\|} \leq \frac{\|u_{\mathrm{reg}}\|}{\|u_{\mathrm{opt}}\|} \sqrt{\kappa(G)}\frac{\varepsilon_{\mathrm{R}}}{1-\varepsilon_{\mathrm{R}}}.
        \end{equation*}
        Since we have shown in the proof of Theorem \ref{thm:projerror} that
        \begin{equation*}
            u_{\mathrm{reg}} = \Pi u_{\mathrm{opt}} + \Psi G^{-1} \Psi^T A(I - \Pi)u_{\mathrm{opt}}
        \end{equation*}
        we can estimate
        \begin{equation*}
            \|u_{\mathrm{reg}}\| \leq \|\Pi u_{\mathrm{opt}}\| + \|\Psi G^{-1} \Psi^T A(I - \Pi)u_{\mathrm{opt}}\| \leq \| u_{\mathrm{opt}}\| + \sqrt{\kappa(A)} \|(I - \Pi)u_{\mathrm{opt}}\|.
        \end{equation*}
        As before, we have used the fact that
        \begin{equation*}
            \Psi G^{-1} \Psi^T A = \Psi V_X \Sigma_X^{-1} U_X^T Y \implies \|\Psi G^{-1} \Psi^T A \| \leq \sqrt{\kappa(A)}.
        \end{equation*}
        From $\|u_{\mathrm{opt}} - \Pi u_{\mathrm{opt}}\| \leq \varepsilon_{\mathrm{P}}\|u_{\mathrm{opt}}\|$ it follows that
        \begin{equation*}
            \frac{\|u_{\mathrm{reg}}\|}{\|u_{\mathrm{opt}}\|} \leq 1+\varepsilon_{\mathrm{P}}\sqrt{\kappa(A)},
        \end{equation*}
        which completes the proof.
    \end{proof}
\end{corollary}
If we assume that $\varepsilon_{\mathrm{P}} \sqrt{\kappa(G)} \approx 1$, then the error estimate from Corollary \ref{cor:totalerror} states, with small leading constants, that
\begin{equation*}
    \frac{\|u_{\mathrm{opt}} - \hat{u}_{\mathrm{reg}}\|}{\|u_{\mathrm{opt}}\|} \leq \mathcal{O}\left((\varepsilon_{\mathrm{R}} +
         \varepsilon_{\mathrm{P}}) \sqrt{\kappa(A)}\right).
\end{equation*}
It therefore makes sense to have a sketching error $\varepsilon_{\mathrm{R}}$ that is of the same order as the projection error $\varepsilon_{\mathrm{P}}$. In practice we found that projection errors of roughly 1\% to 10\% can be expected so that the sketching induced error isn't very harmful if we choose the sample size as in Theorem \ref{thm:randerror} with $\varepsilon_{\mathrm{R}} = 0.1$.

\section{Numerical results}\label{sec:ne}

To test the performance of Algorithm \ref{alg:sketch1} we consider the finite element formulation of the elliptic equation \eqref{pde} with homogeneous Dirichlet boundary conditions $u=0$ on $\partial \Omega$ and a forcing term derived from a piecewise constant approximation of the function
$$
f(x) = \begin{cases} 
5 & \text{if} \; \sqrt{(x_1+\frac 1 2)^2 + x_2^2 + x_3^2} \leq 0.3,\\
0 & \text{otherwise},
\end{cases}.
$$
We discretise the model on a spherical domain  $\Omega$ $(d=3)$ of unit radius comprising $k=684560$ unstructured linear tetrahedral elements. This leads to a total $116805$ nodes of which $n=101509$ are situated in the interior of the domain. In these circumstances $X$ is a tall matrix with $2053680$ rows, the stiffness matrix $A$ has dimensions $101509 \times 101509$ and the sample space is $[2053680]$.  

We seek to assess the practical performance of our algorithm in terms of its speed and accuracy in computing the sketched solution under various choices sampling budgets and low-dimensional subspaces, for the proposed sampling distribution. To achieve this we perform three benchmark tests involving realisations of (i) a uniformly distributed random parameter field, (ii) a smoothly varying lognormal random field, and (iii) a random field with jump discontinuities. For each of these we run a sequence of $N=100$ simulations, i.e. $p$ queries, and record timings and error measures on average. For each realisation we compute also the conventional FEM solution to provide a reference for comparison. The high-dimensional $u_{\mathrm{opt}}$ is computed using Matlab's built-in \verb"A\b" command \cite{MATLAB:18}, and the times provided include the efficient assembly of the full stiffness matrix as a triple product of sparse matrices $A=D^TZ^2D$. Our code was implemented in Matlab R2018b and executed on a workstation equipped with two 14-core Intel Xeon dual processors, running Linux NixOS with 384GB RAM.

In the offline phase of Algorithm \ref{alg:sketch1} we form a low-dimensional ONB for the projection by computing the last eigenfunctions of the sparse Laplacian matrix discretised on $\Omega$. For this time consuming and memory demanding operation we have resorted to the \verb"svds" and \verb"qr" commands which avoid computing the complete spectrum or they produce a sparse ONB respectively. The computation of the sampling distribution based on the leverage scores of $X_{\Delta} = Z_{\Delta} D \Psi$ was also performed once during the offline phase and took about 4 hours, using the \verb"svd(,'econ')" command. The  distribution $q$ was sampled with replacement during the online phase of the algorithm using the efficient  command \verb"datasample", which indicatively, for the chosen $q$, outputs a million samples in less than 0.3 s. Notice that although this sampling implementation is not independent of the dimension $kd$, there exist alternative schemes that can handle arbitrarily large distributions with constant complexity \cite{Bringmann2017}.

In the implementation of the algorithm we record the following quantities--diagnostics that provide evidence on the performance in the conditions of each benchmark: the ratio $c'/3k$ indicating how many of the rows of $X$ are used in the sketch, the relative subspace projection error $\|\Pi u_{\mathrm{opt}} - u_{\mathrm{opt}}\|/\|u_{\mathrm{opt}}\|$, the upper bound of the randomisation error $\|\hat G^{-1}G - I\|$, the relative regression error $\|\hat u_{\mathrm{reg}} - u_{\mathrm{reg}}\|/\|u_{\mathrm{reg}}\|$, and the relative total error $\|\hat u_{\mathrm{reg}} - u_{\mathrm{opt}}\|/\|u_{\mathrm{opt}}\|$. In the context of real-time model prediction in manufacturing processes an upper limit of 10\% for the total error is deemed reasonable. 

\subsection{Uniformly random parameter field}\label{urpf}

In this first instance we simulate sketched solutions for 100 parameter vectors $p \in \mathbb{R}^k$ drawn at random from $\mathcal{U}\bigl ([10^{-1},10^2]\bigr)$. Five sets of simulations were performed using ONBs incorporating the last $\rho=\{50, 100\}$ singular functions of the Laplacian. Our focus was on monitoring the trade-off between accuracy and time consumption when  $c=\{5 \times 10^5, 10^6, 5\times 10^6\}$ iid samples are drawn from $p$. The results are tabulated in table \ref{tableUN}.  

Although the values in $p$ vary over four orders of magnitude, the parameter has a homogeneous expectation within the domain and thus overall the algorithm yields sketched solutions at 10\% or less total error, with only 100 basis functions. The results show that the sampling is highly non-uniform since even in the case where a million idd samples were taken these involved only 41074, a mere 6\%, of the rows of $X$. The sketching-induced error factor $\|\hat G^{-1}G - I\|$ appears to reduce almost linearly with the number of samples $c$. Comparing the  relative subspace projection $\|\Pi u_{\mathrm{opt}} - u_{\mathrm{opt}}\|$ and total $\|\hat u_{\mathrm{reg}} - u_{\mathrm{opt}}\|$ errors note that for $\|\hat G^{-1}G - I\| \approx 1$ the later is kept marginally larger than the former, which verifies the regularising effect of the projection on the sketching-induced noise. It is also important to see that in switching from $\rho=50$ to $\rho=100$ the projection error is halved to 0.03, however the number of samples necessary to yield the same levels of the error increases by about 5 times. For relative error tolerances around the 10\% mark, the times recorded are below 1 s, while by comparison the time for computing $u_{\mathrm{opt}}$ was on average found to be 23.75 s. 

The trade-off between speed and accuracy can be seen by comparing the results in the first and last rows of the table \ref{tableUN} where the algorithm achieves a 4\% total error, when the projection error is at 3\%, after five million samples. On the other hand, solutions within a 10\% error margin, when the projection error is at 7\%, are obtained in less than 0.5 s, which is 55 times faster than computing the standard $u_{\mathrm{opt}}$. The speedup in sketching the more accurate solution with $\rho=100$ and $c=5$ million is still 7 times faster, compared to the FEM solver. The histograms in figure \ref{fig:histograms} provide a further insight on how the various error components vary within the ensemble of the 100 problems. We point out that the numerical results are in good agreement with the assertion of Theorem \ref{thm:randerror}. For the example shown in figure \ref{fig:histograms}, i.e. when $\rho=50$ and the error tolerance is $\varepsilon = 10\%$, our theorem predicts $c=15 \rho \log(15 \rho)\beta^{-1} \varepsilon^{-2} \approx 5.0\cdot 10^5\beta^{-1} $ samples which is consistent to the observed $c=1$ when $\beta^{-1} \approx 2$. In the histograms we see that the sketching error virtually never exceeds $10\%$ and that $\|\hat{G}^{-1}G - I\|$ exhibits the same pattern as $\|u_{\mathrm{opt}}- u_{\mathrm{reg}}\| / \|u_{\mathrm{opt}}\|$ which supports the claim that this quantity is driving the sketching error. Similar observations can be made for the other cases of table \ref{tableUN}. Figure \ref{fig:histograms} also shows that, although their magnitude is comparable, the variability in the projection error is much smaller than that of the sketching error. This is not surprising as the sketching is an intrinsically random method while the differences in the projection are only due to perturbations in the parameter.

\begin{table}
\centering
\begin{tabular}{c|c|c|c|c|c|c|c}
  \hline
$\rho$ & $c$ [$10^6$] & time [s] & $c'/3k$ & $\frac{\|\Pi u_{\mathrm{opt}} - u_{\mathrm{opt}}\|}{\|u_{\mathrm{opt}}\|}$  & $\|\hat G^{-1} G - I\|$ & $\frac{\|\hat u_{\mathrm{reg}} - u_{\mathrm{reg}}\|}{\|u_{\mathrm{reg}}\|}$ & $\frac{\|\hat u_{\mathrm{reg}} - u_{\mathrm{opt}}\|}{\|u_{\mathrm{opt}}\|}$ \\ \hline 
50 & 0.5 & 0.43 & 0.04 & 0.07 & 1.60 & 0.07 & 0.09\\
50  & 1 & 0.78 &  0.06 & 0.07 & 1.07 & 0.05 & 0.08\\
100 & 0.5 & 0.49 & 0.04 & 0.03 & 3.99 & 0.11 & 0.11\\
100 & 1 & 0.80 & 0.06 & 0.03 & 2.30 & 0.06 & 0.07\\
100 & 5 & 3.22 & 0.11 & 0.03 & 0.77 & 0.02 & 0.04\\
\hline
\end{tabular}
\caption{Numerical results for the tests performed with $p \sim \mathcal{U}([10^{-1},10^2])$. The quantities above are averages over 100 runs with different $p$ realisations. The results show the impact of $c$ and $\rho$ on the various error components and the computing times. Note that for a sufficiently large $c$ the total error is only marginally larger than the projection error, which manifest the regularising effect of the projection on the sketching induced error.}
\label{tableUN}
\end{table}

\begin{figure}
    \centering
    \includegraphics[width=0.5\textwidth]{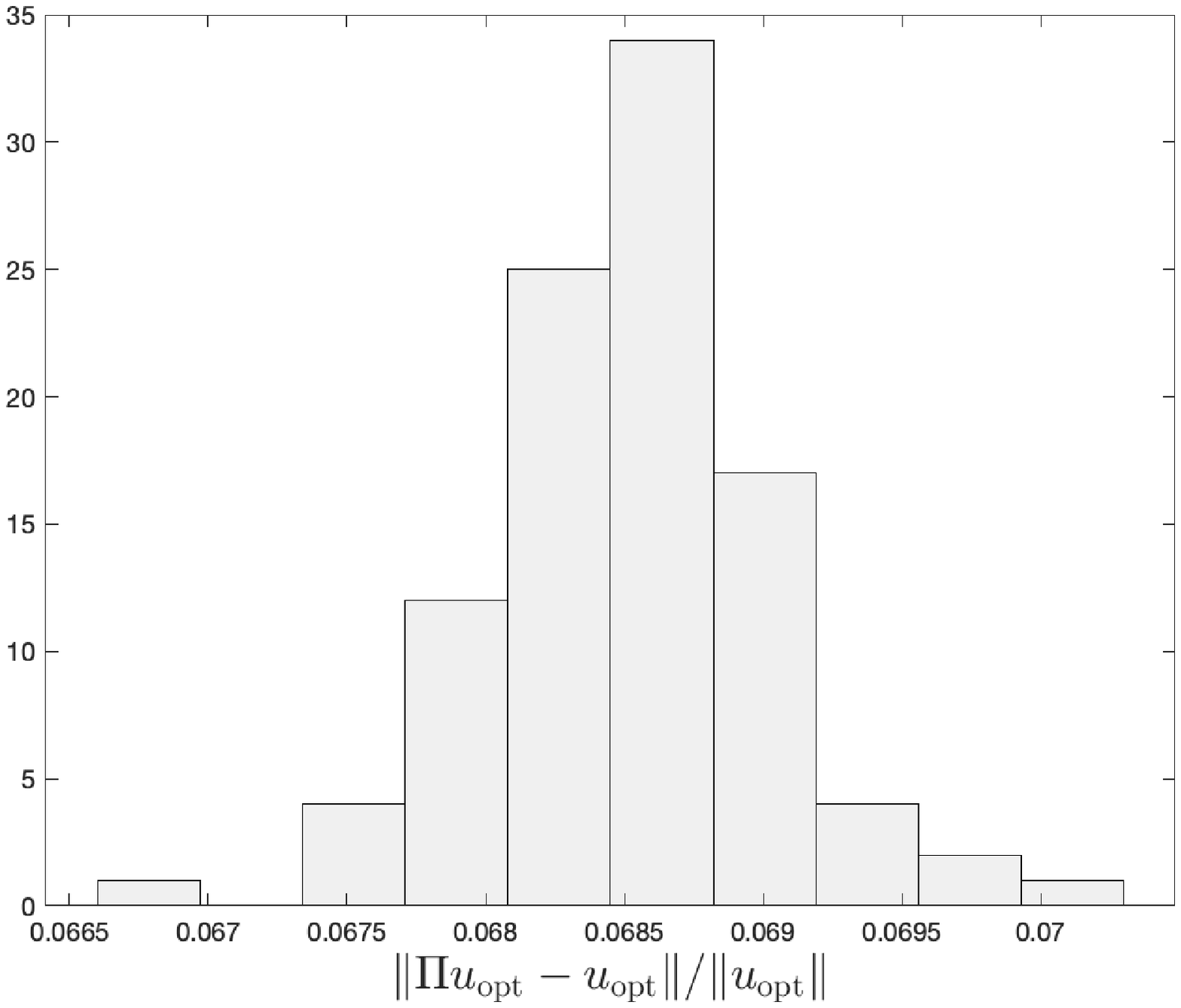}\includegraphics[width=0.5\textwidth]{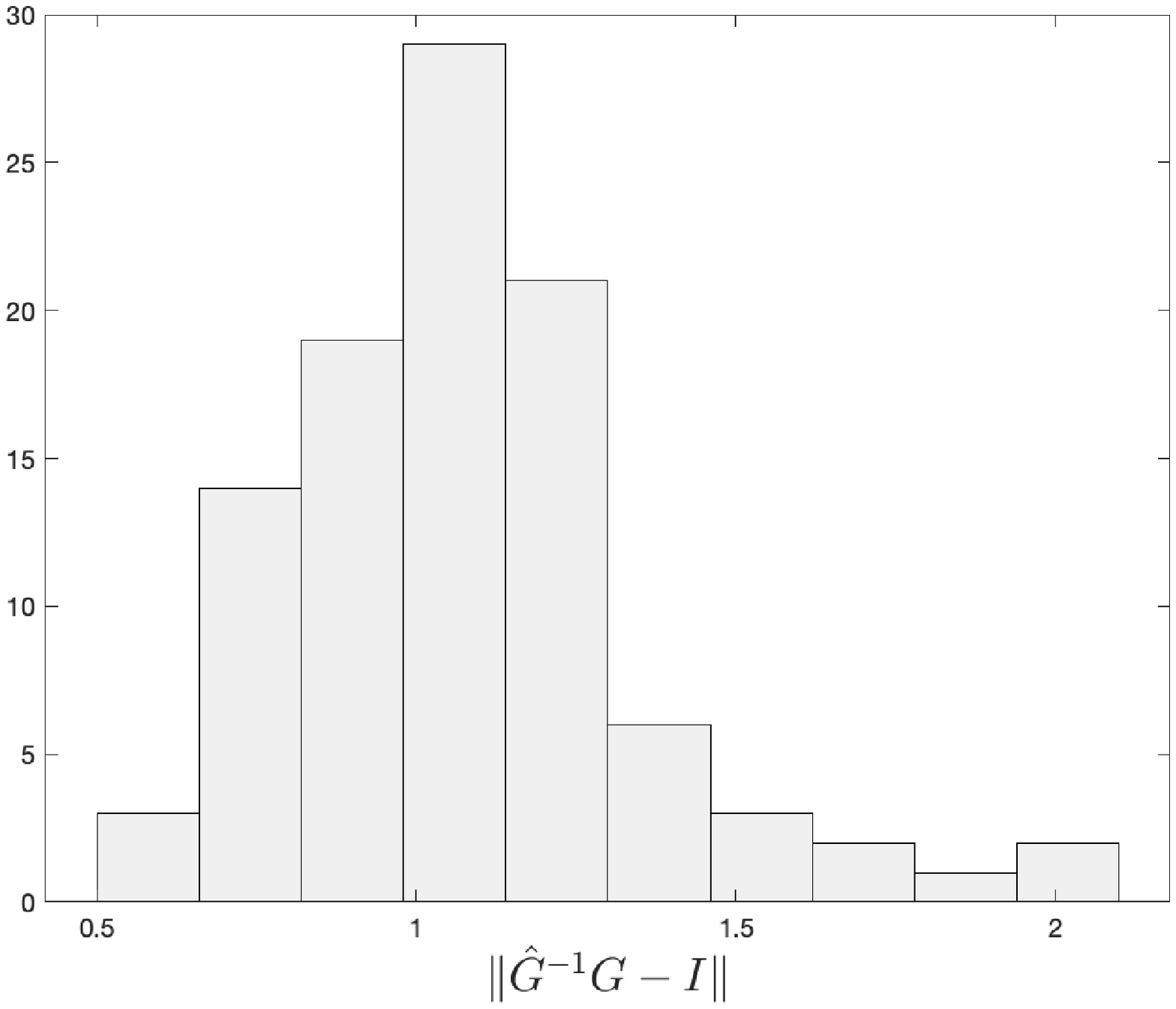}\\
    \includegraphics[width=0.5\textwidth]{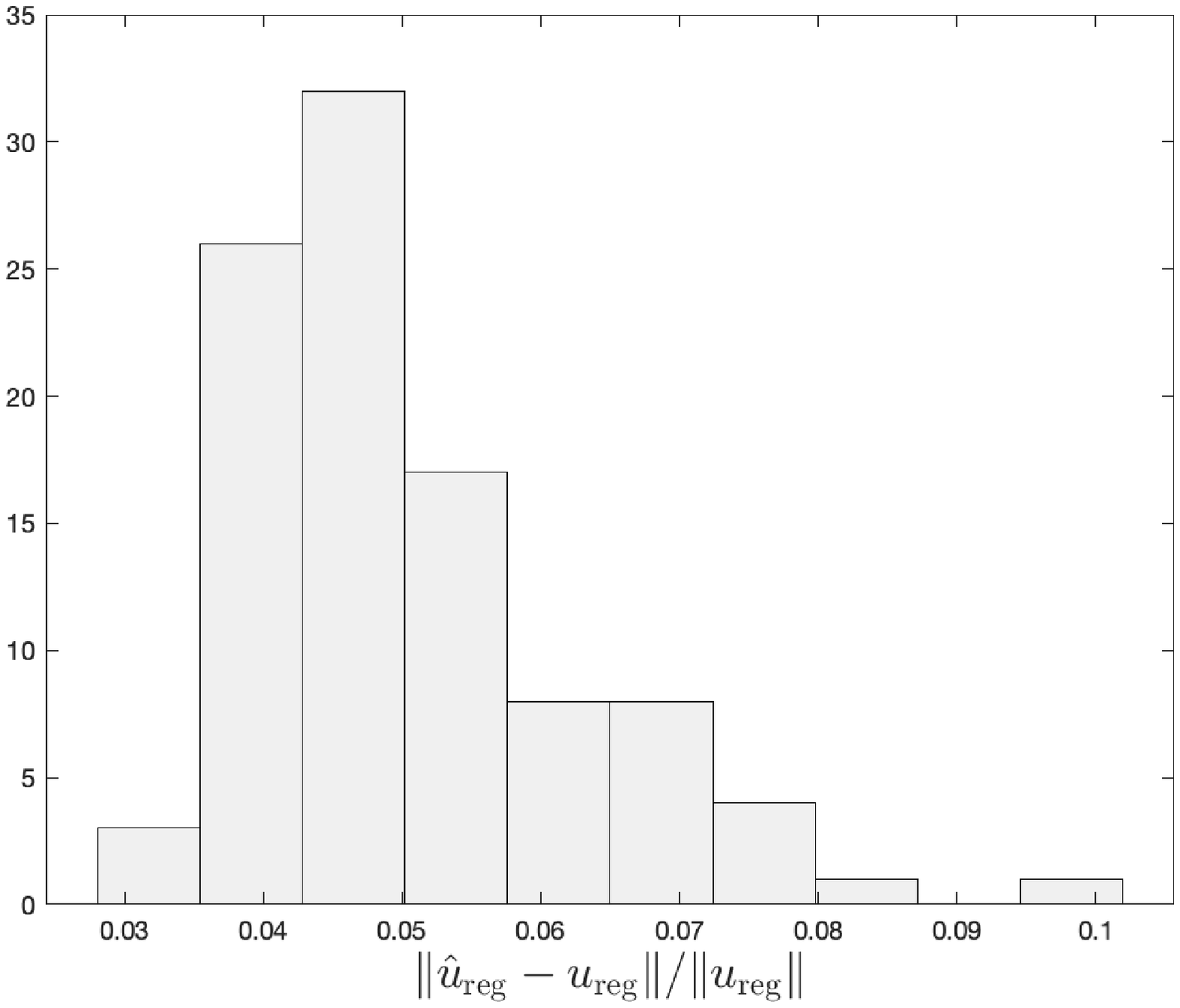}\includegraphics[width=0.5\textwidth]{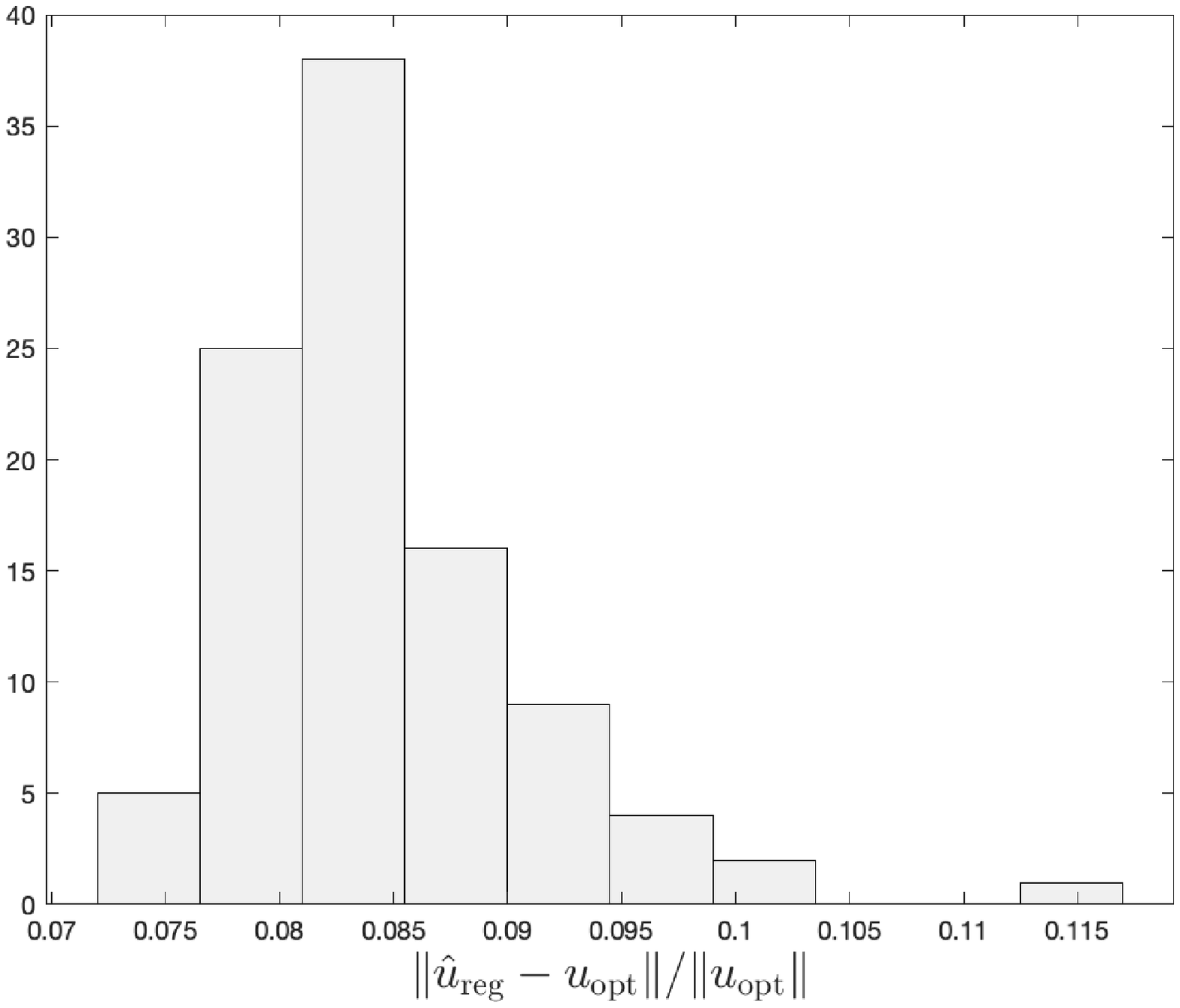}
    \caption{Histograms showing the variation in the various error quantities relating to the performance of our algorithm, as recorded in the table \ref{tableUN} for 100 different realisations of the $p$ vector from $\mathcal{U}([10^{-1},10^2])$ of the code with $\rho=50$ and $c=1$ million.}
    \label{fig:histograms}
\end{figure}

\subsection{Smooth parameter field}\label{spf}

In the second benchmark we turn our attention to parameter functions with smooth spatial variation like those encountered in the context of uncertainty quantification for PDEs \cite{LordPowellShardlow}. As the anticipated FEM solution is smooth we maintain the bases used in \ref{urpf}. In this case, the parameter $p$ is a lognormal random field given by $p \doteq \exp(b)$, where $b$ is a zero-mean Gaussian random field with Whittle-Mat\'{e}rn covariance function with smoothness parameter $\nu>0$ given by
\be
C_{b}(x,y)=\frac{\operatorname{Var}[b]}{2^{\nu-1} \Gamma(\nu)}\left(\|x-y\|_M\right)^{\nu} K_{\nu}\left(\|x-y\|_M\right), \quad x, y \in \Omega,
\ee
where $\Gamma(\nu)$ is the Gamma function, $\|x\|_{M}^{2}=x^{T} M^{-1} x$ is the weighted Euclidean norm with positive definite matrix $M$ and $K_{\nu}$ is the order $\nu>0$ modified Bessel function of the second kind. Here we use $\nu=15/2$, $M^{1/2}=\mathrm{diag}(1/5,1/5,1/5)$ and $\operatorname{Var}[b]=1$. We draw realisations of $p$ by calculating once the Karhunen-Lo\`{e}ve expansion of $b$ and then drawing iid from $\mathcal{N}(0,1)$.

The results presented in table \ref{tableWM} show a similar performance to the uniformly random case in subsection \ref{urpf}. The suitability of the low-dimensional subspace is evidenced by the 7\% relative projection error attained at $\rho=50$. Sketched solutions within an error tolerance of 10\% were computed in less than 1 s. Further, note that the total error is within a 2\% margin from the projection error, which demonstrates the effectiveness of our sketching regularisation approach, apart from the test with $\rho=100$ and $c=1$ where $\|\hat G^{-1}G - I \|$ is considerably higher, implying that $c$ was insufficiently small for that test. This observation is consistent with our error bound in  \eqref{thm:randerror}.
Comparing the results for $(\rho=50,c=5)$ and $(\rho=100,c=1)$ shows that in the former case, although using half the number of basis functions and five times more samples, due to the larger projection error, the total error is still 1\% larger than that of the later. 
The images presented in figure \ref{fig:images} correspond to one of the simulations in this benchmark with $\rho=100$ and $c=1$ million, illustrating a cross section of the profile of $p$, the exact FEM solution, the sketched solution and the relative error between the two.

\begin{table}
\centering
\begin{tabular}{c|c|c|c|c|c|c|c}
  \hline
$\rho$ & $c$ [$10^6]$ & time [s] & $c'/3k$ & $\frac{\|\Pi u_{\mathrm{opt}} - u_{\mathrm{opt}}\|}{\|u_{\mathrm{opt}}\|}$  & $\|\hat G^{-1} G - I\|$ & $\frac{\|\hat u_{\mathrm{reg}} - u_{\mathrm{reg}}\|}{\|u_{\mathrm{reg}}\|}$ & $\frac{\|\hat u_{\mathrm{reg}} - u_{\mathrm{opt}}\|}{\|u_{\mathrm{opt}}\|}$\\ 
\hline 
25  & 0.5 & 0.52 & 0.04 & 0.15 & 0.73 & 0.05 & 0.17\\
50  & 1 & 0.52 &  0.06 & 0.07 & 0.95 & 0.04 & 0.08\\
50  & 5 & 3.51 & 0.12 & 0.07 & 0.35 & 0.02 & 0.07\\
100 & 1 & 0.85 & 0.06 & 0.03 & 1.97 & 0.05 & 0.06\\
100 & 5 & 3.51 & 0.12 & 0.03 & 0.65 & 0.04 & 0.04\\
\hline
\end{tabular}
\caption{Numerical results for the tests with lognormal random field drawn from a Whittle-Mat\'{e}rn model with a smooth covariance. The algorithm yields solutions with less than 10\% error with as few as 50 basis functions. Similar to the uniformly random case in table \ref{tableUN}, the total errors are sustained close to the projection errors when $\|\hat G^{-1}G - I\| < 1$.}
\label{tableWM}
\end{table}

\begin{figure}
\centering 
\includegraphics[width=0.45\linewidth]{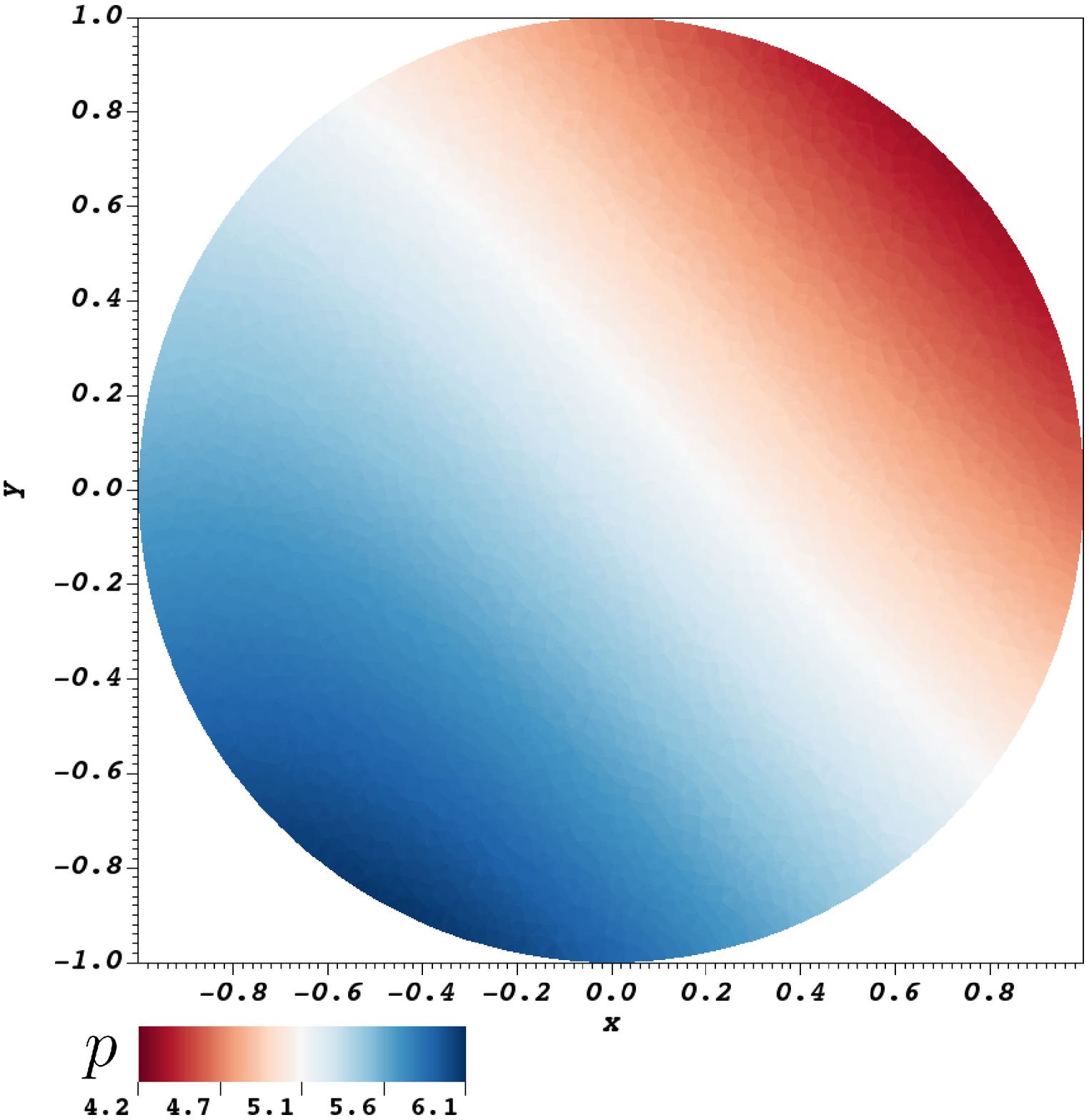}
\includegraphics[width=0.45\linewidth]{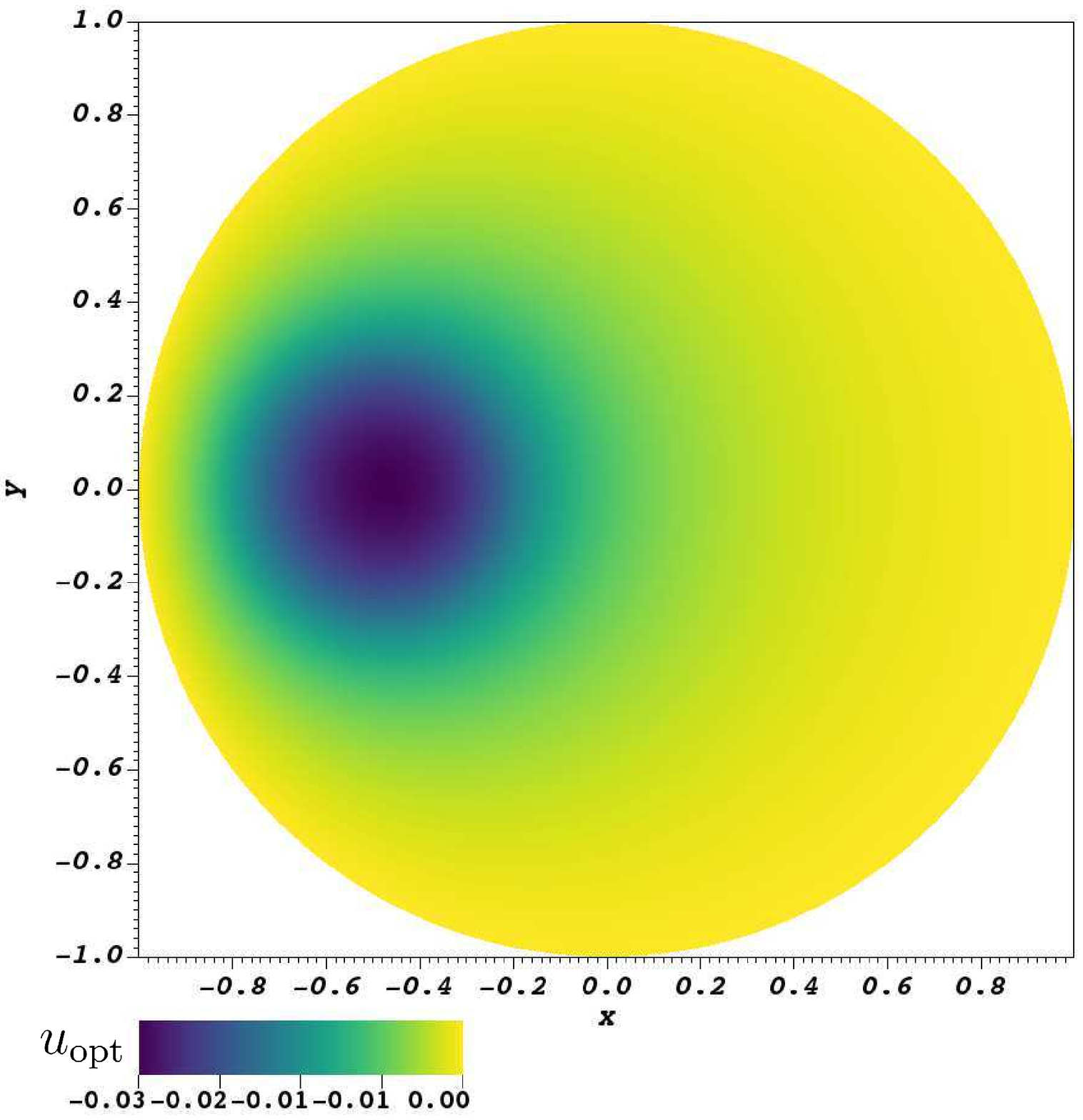}\\
\includegraphics[width=0.45\linewidth]{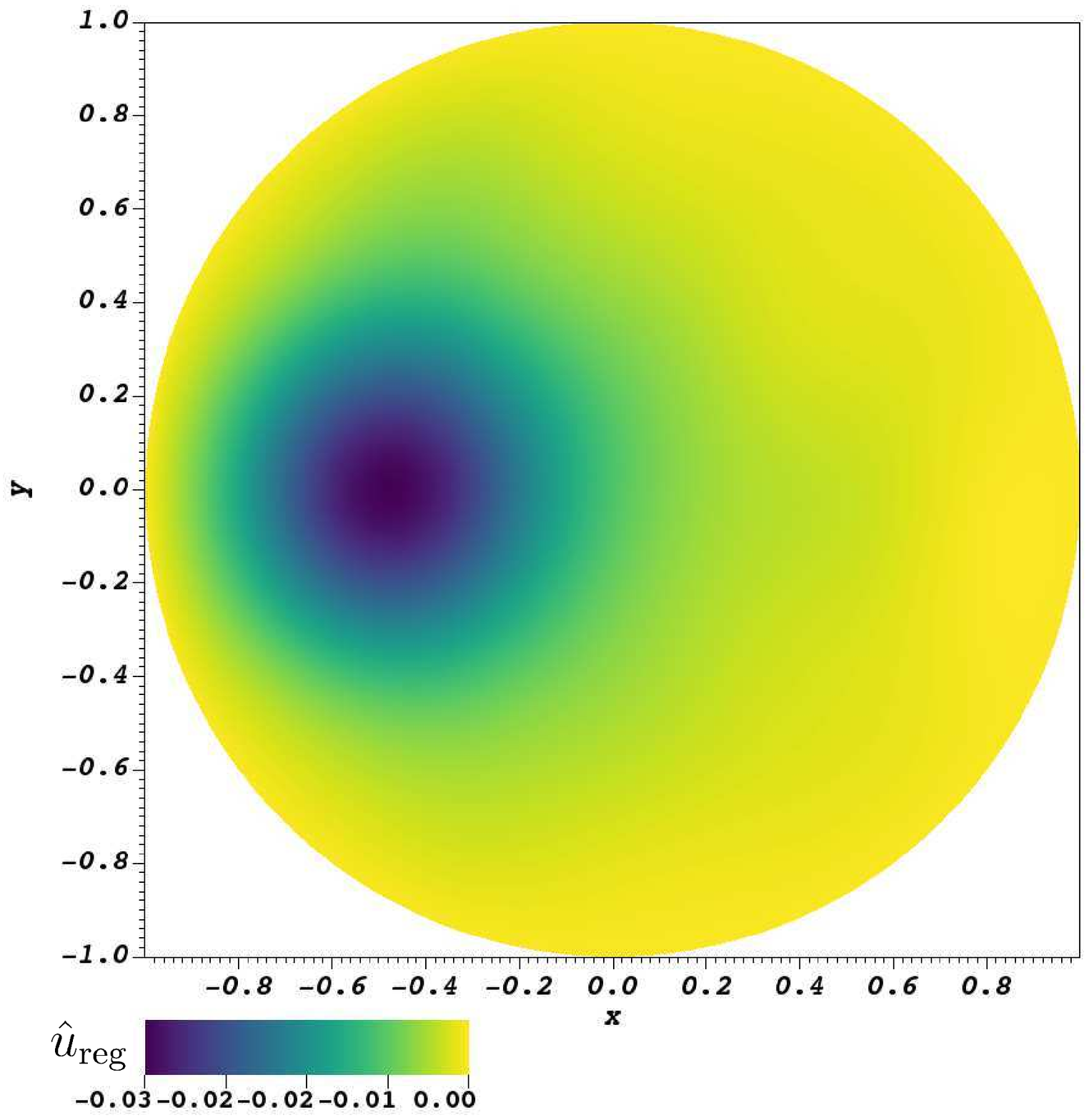}
\includegraphics[width=0.45\linewidth]{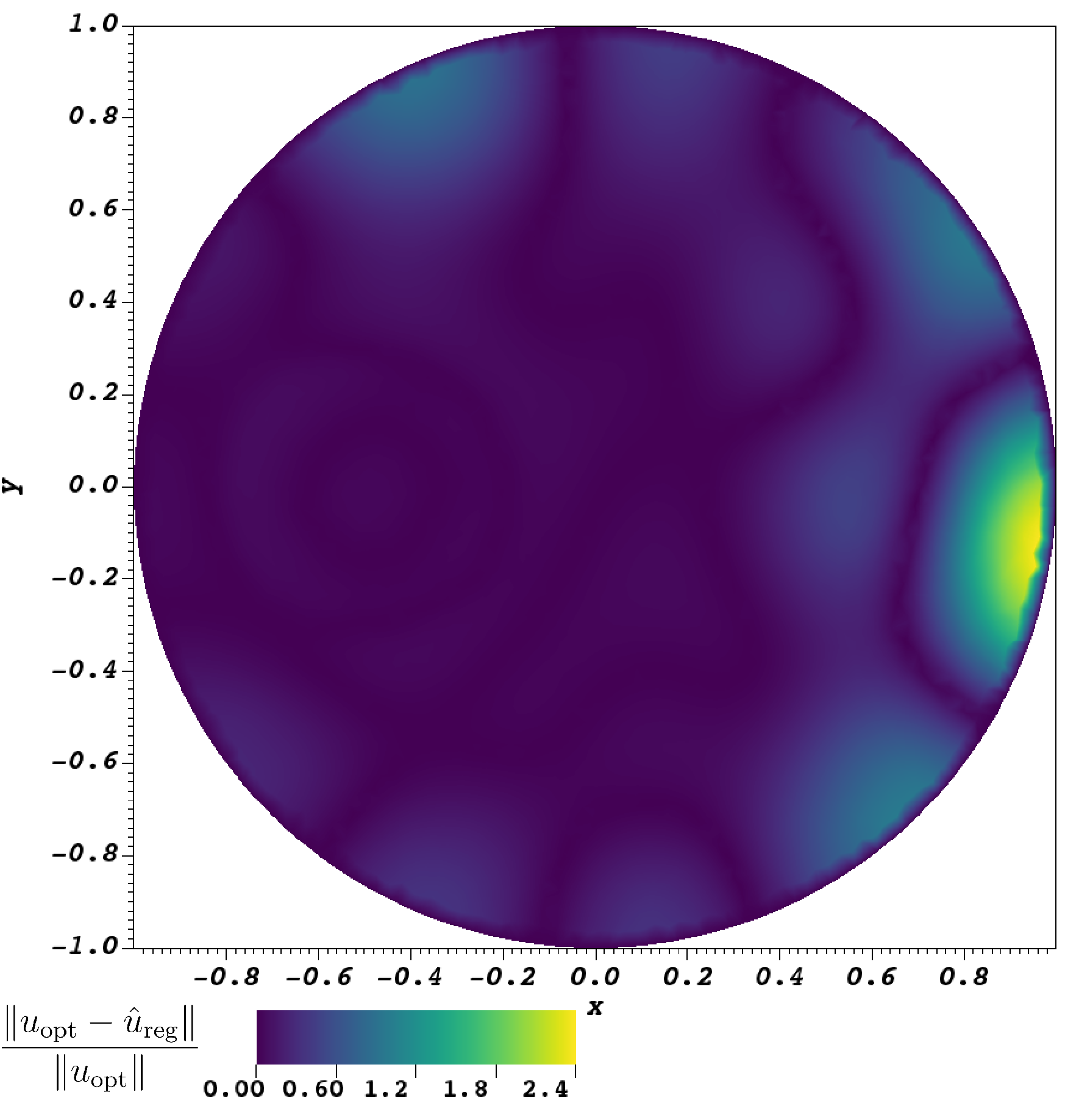}
\caption{At the top left, a view of a lognormal field $p$ sampled from the Whittle-Matt\'{e}rn class, and to its right the corresponding view of $u_{\mathrm{opt}}$ that took 23.75 s to compute. Below to the left, the sketched projected solution $\hat u_{\mathrm{reg}}$ computed after 0.83 s and to its right the profile of the relative error between $u_{\mathrm{opt}}$ and $\hat u_{\mathrm{reg}}$.}
\label{fig:images}
\end{figure}

\begin{figure}
\centering 
\includegraphics[width=0.65\linewidth]{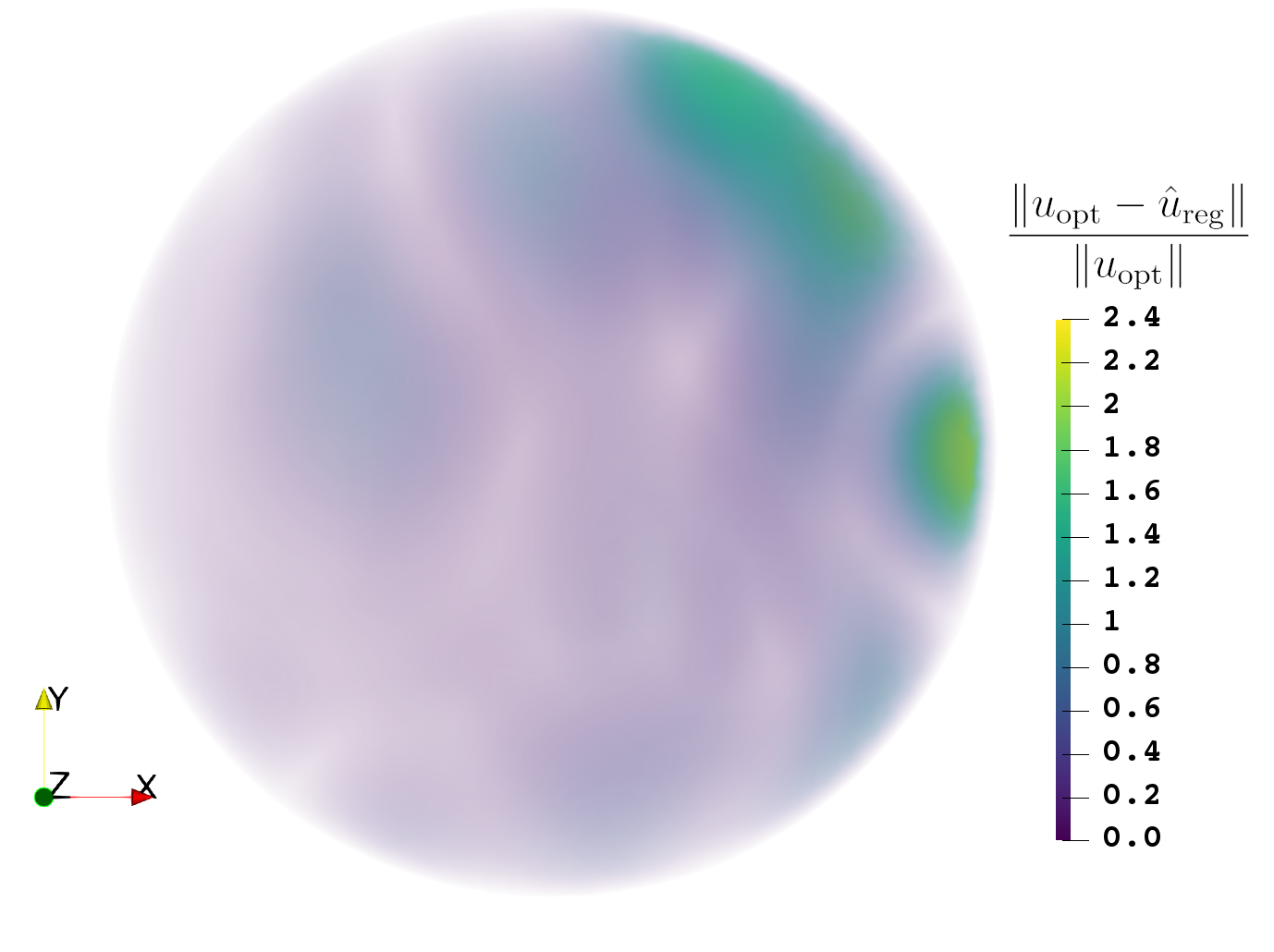}
\caption{A 3D view of the relative error profile between $u_{\mathrm{opt}}$ and $\hat u_{\mathrm{reg}}$.}
\label{fig:image3d}
\end{figure}

\subsection{Non-smooth parameter field}

A more challenging benchmark test is to consider the FEM solution for a parameter field with non-smooth variation. In this case it is natural to anticipate that any significant jump discontinuities in the profile of $p$ will have an adverse effect on the condition number of the stiffness matrix  \cite{KamenskiHuangXu}. For our simulations we choose a piecewise constant approximation of the positive function 
$$
p(x) \doteq 9.1+\mathrm{sgn}(x_1) + 3\mathrm{sgn}(x_2) + 5\mathrm{sgn}(x_3) + 0.1\mathcal{U}\bigl ( [0,1]\bigr )
$$ 
which is discontinuous along the three axes. The sign function $\mathrm{sgn}:\mathbb{R} \to \mathbb{R}$ is given by $\mathrm{sgn}(x)=x/\lvert x \rvert$ when $x \neq 0$ and $\mathrm{sgn}(0)=0$. In constructing the projection subspace we found that the smooth basis utilised in the previous cases was not appropriate to this case and we thus resorted in a sparse ONB taking a subset of the columns of the sparse unitary matrix computed from the QR decomposition of the Laplacian. 

The results in table \ref{tableDC} suggest that the chosen basis is not very appropriate since not only the number of basis functions is substantially larger, but also the reduction in the projection error for a 100\% increase in $\rho$ is quiet marginal. In turn, this increase in the dimension of $\hat G$ affects the level of sketching error, as even with $c=5$ million samples $\|\hat G^{-1} G - I\| > 1$. Consequently, this has a profound effect on timings, although the sketched approach maintains a five fold advantage to the standard FEM solver. For the tests for $(\rho=2 \times 10^3, c=10^6)$ and $(\rho=2 \times 10^3, c=5 \times 10^6)$ notice that increasing the samples by five times does not yield a significant improvement in the results, which is likely triggered by the large $\kappa(A) \approx 10^5$ in the error term of Theorem \ref{thm:projerror} which causes the $\|u_{\mathrm{reg}} - u_{\mathrm{opt}}\|$ to grow. 
\begin{table}
\centering
\begin{tabular}{c|c|c|c|c|c|c|c}
\hline
$\rho$ & $c$ $[10^6]$ & time [s] & $c'/3k$ & $\frac{\|\Pi u_{\mathrm{opt}} - u_{\mathrm{opt}}\|}{\|u_{\mathrm{opt}}\|}$  & $\|\hat G^{-1} G - I\|$ & $\frac{\|\hat u_{\mathrm{reg}} - u_{\mathrm{reg}}\|}{\|u_{\mathrm{reg}}\|}$ & $\frac{\|\hat u_{\mathrm{reg}} - u_{\mathrm{opt}}\|}{\|u_{\mathrm{opt}}\|}$ \\\hline 
1000 & 1 & 2.67 & 0.06 & 0.07 & 4.61 & 0.01 & 0.26 \\
1000 & 5 & 5.96 & 0.12 & 0.05 & 1.25 & 0.01 & 0.26\\
2000 & 1 & 4.87 & 0.06 & 0.02 & 77.36 & 0.02 & 0.08\\
2000 & 5 & 9.95 & 0.12 & 0.03 & 9.64 & 0.01 & 0.08\\
\hline
\end{tabular}
\caption{Numerical results for the non-smooth parameter field. In this case the algorithm requires a far more extensive basis, and thus considerably more samples and computing time to yield solutions within the required 10\% error margin.}
\label{tableDC}
\end{table}

\section{Conclusions}
We have considered expediting the solution of the finite element method equations arising from the discretisation of elliptic PDEs on high-dimensional models. Taking into consideration the multi-query context and the smooth profile of the FEM solution, we proposed a practical sketch-based algorithm that involves projection onto lower-dimensional subspace and sketching using a generic,  sampling distribution derived from the leverage scores of a tall matrix associated with the Laplacian operator. We have elaborated on the impact of the projection in reducing the dimensionality as well as mitigating the effects of sketching noise. The performance of our method was evaluated in a series of  benchmark tests of FEM simulations that demonstrated substantial speed improvements at the cost of a small compromise in accuracy when the stiffness matrix is well conditioned.

\end{document}